\def\version{\today}

\documentclass[reqno]{amsart}
\usepackage[usenames,dvipsnames]{xcolor}
\usepackage{amssymb,epsfig,graphics,graphicx,bbm,hyperref}
\usepackage{mathabx}
\usepackage{pgf,pgfarrows,pgfnodes,pgfautomata,pgfheaps,pgfshade}

\usepackage{multicol}
\usepackage{enumerate}
\usepackage[usenames,dvipsnames]{xcolor}
\usepackage{tikz}
\usepackage{nicefrac}

\definecolor{gray}{rgb}{0.93,0.93,0.93}
\definecolor{light-gold}{rgb}{0.99,0.97,0.78}

\definecolor{gold}{rgb}{0.7,0.55,0}

\def\be{\begin{equation}}
\def\ee{\end{equation}}
\def\bm{\begin{multline}}
\def\bfig{\begin{figure}[htb]}
\def\efig{\end{figure}}

\newcommand{\eqd}{\overset{\rm d}{=}}

\numberwithin{equation}{section}
\newtheorem{theorem}{Theorem}[section]
\newtheorem{proposition}[theorem]{Proposition}
\newtheorem{lemma}[theorem]{Lemma}

\newtheorem{remark}[theorem]{Remark}

\newcommand{\eps}{{\varepsilon}}

\def\ba#1\ea{\begin{align*}#1\end{align*}}
\def\ban#1\ean{\begin{align}#1\end{align}}

\newcommand{\bs}{\boldsymbol}
\newcommand{\sss}{\scriptscriptstyle}

\makeatletter
  \def\tagform@#1{\maketag@@@{\scriptsize{(#1)}\@@italiccorr}}
\makeatother

\renewcommand{\eqref}[1]{(\ref{#1})}

\newcommand{\EE}{\mathbb{E}} 

\newcommand{\PP}{\mathbb{P}}

\newcommand{\oo}{\infty}

\newcommand{\bb}[1]{\mathbb{#1}}
\renewcommand{\c}[1]{\mathcal{#1}}
\newcommand{\f}[1]{\mathfrak{#1}}

\begin{document}

{\hfill\small \version} \vspace{2mm}

\title{A random recursive tree model with doubling events}

\author{Jakob E. Bj\"ornberg}
\address{Department of Mathematics,
Chalmers University of Technology and the University of Gothenburg,
Sweden}
\email{jakob.bjornberg@gu.se}
 
\author{C\'ecile Mailler}
\address{Department of Mathematical Sciences, University of Bath, Bath BA2 7AY, United Kingdom}
\email{c.mailler@bath.ac.uk}


\maketitle

\begin{abstract}
We introduce a new model of random tree that grows like a random recursive tree, except at some exceptional ``doubling events'' when the tree is replaced by two copies of itself attached to a new root.
We prove asymptotic results for the size of this tree at large times, its degree distribution, and its height profile. We also prove a lower bound for its height.
Because of the doubling events that affect the tree globally, 
the proofs are all much more intricate than in the case 
of the random recursive tree in which the growing operation is always local.
\end{abstract}

\section{Introduction}

\subsection{Model and motivation}

In this paper
we consider a variant of the random recursive tree, with what we call
\emph{doubling events}.
Recall that the random recursive tree
is a  process of growing random trees where,
at each step, a new leaf is added to a node
chosen uniformly at random, starting
from a single root node.
In our  process, 
we also randomly grow a tree by selecting 
a uniformly random node at each step,
and similarly to the random recursive tree, a leaf is  added to that node
\emph{unless} it is the root of the tree.  If the randomly chosen node is
the root, however, a doubling event occurs, which means
that  we
replace the entire subtree of the root with two
copies of itself.  See Figure \ref{fig:rdt} for an illustration;
a more formal description of the
process, using the Ulam--Harris framework, is given below.

This model was introduced to us by Olivier Bodini, 
who asked whether we could get some information on the size of the tree at large times 
(one of our main results is to show that it is linear in the number of steps). 
Bodini sees this model as a simplification for a more intricate model in which, 
at every time step, we pick a node $\nu$ uniformly at random and replace 
it by a new node whose two subtrees are two copies of the tree rooted at $\nu$.
In other words, doubling events happen not only at the root, but
everywhere in the tree. 
Bodini would eventually like to understand the size of this doubling
tree after $n$ steps: in our last result we prove that, in
expectation, this size is 
superlinear in~$n$ (see Proposition~\ref{prop:double_everywhere}). 

\begin{figure}[b]
\centering\includegraphics[scale=1.7]{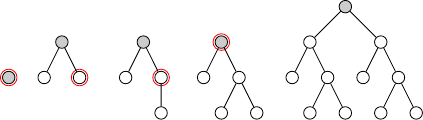}
\caption{Steps in the construction of a random doubling tree.
  The root is drawn grey and, at each
step, the randomly selected node is circled in red.  In the
first and fourth steps, the selected node is the root and a doubling
event occurs.  At the other steps, a non-root node is selected and
a leaf is added to that node.}\label{fig:rdt}
\end{figure}

\medskip
We now recall the Ulam--Harris notation for trees, and define our
process using this framework.
A tree~$\tau$ is a set of finite words using the alphabet 
$\{1, 2,3,\ldots\}$ such that, for all $w\in \tau$, all prefixes of $w$ are also
in $\tau$.  That is to say, if $w=w_1w_2\dotsc w_m\in\tau$,
with each $w_i\in\{1, 2,3,\ldots\}$, then 
 $w_1w_2\dotsc w_k\in\tau$
for all $k\in\{0,1,\dotsc,m\}$.
Each element $w\in\tau$ is called a
\emph{node} or \emph{vertex} of $\tau$, 
and the empty word $\varnothing$
is called the \emph{root}.
The number of nodes in a tree $\tau$ is denoted $|\tau|$,
while if $w_1w_2\dotsc w_m\in\tau$ is a node
then $|w|:=m$ denotes its length as
a word, and is also called its \emph{height}.
One can see a tree as a genealogical structure: the prefixes of a word
are its ancestors, the longest of its prefixes is its parent, the
other children of its parents are its siblings, etc. 

We now formally define the random recursive tree with
doubling events.
We define the sequence of random trees $(\tau_n)_{n\geq 1}$ by setting 
$\tau_0 = \{\varnothing\}$ and for all $n\geq 0$, given~$\tau_n$,  
\begin{itemize}
\item we pick a node (i.e.\ word) $\nu_n$ uniformly at random among
  the nodes of $\tau_n$, 
\item if $\nu_n = \varnothing$, then we set
$\tau_{n+1} = \{\varnothing\}\cup
\{1w \colon w\in \tau_n\}\cup \{2w  \colon w\in \tau_n\}$,
\item if $\nu_n\neq \varnothing$, then we set
$\tau_{n+1} = \tau_n\cup\{\nu_nj\}$,
where  $j = \min\{i\geq 1\colon \nu_ni\notin \tau_n\}$.
\end{itemize}
In words, this means that each time the randomly chosen node 
$\nu_n$ is the root $\varnothing$, then 
$\tau_{n+1}$ is the tree whose root has two children at which
are attached two copies of $\tau_n$, while if $\nu_n$ is not the root
then $\tau_{n+1}$ equals $\tau_n$ with  one child added to
$\nu_n$.

\subsection{Main results}

Our first result gives an estimate for the size of the tree as time goes to infinity: 
for all $n\geq 0$, we let $B_n$ be the number of non-root nodes in the tree ($B_0 = 0$) at time~$n$.
Note that $|\tau_n| = B_n +1$, for all $n\geq 0$.
Parts (b), (c) and (d) in the following result are simple consequences of part (a); we only state them for clarity.

\begin{proposition}[Asymptotic size]\label{prop:cvBn}\phantom{mmm}
\begin{enumerate}
\item[(a)] For all $k\geq 1$,
\begin{equation}\label{eq:cv_moments}
\mathbb E[(B_n/n)^k] \to m_k
:= \prod_{i=1}^k \bigg(1-\frac1i+\frac{2^i}{i}\bigg) 
= \frac{2^{k(k-1)/2}}{k!}\prod_{i=1}^k \Big(1+\frac{i-1}{2^i}\Big).
\end{equation}
\item[(b)] The sequence  $(B_n/n)_{n\geq 1}$ is tight.
\item[(c)] For any sequence $(\omega(n))_{n\geq 0}$ such that $n=
  o(\omega(n))$, $B_n/\omega(n) \to 0$ in probability as
  $n\uparrow\infty$.  
\item[(d)] For any $\eta>0$, $B_n/n^{1+\eta}\to 0$ almost surely as
$n\uparrow\infty$.
\end{enumerate}
\end{proposition}

\begin{remark}
This result comes close to, but does not quite reach,
establishing weak convergence of the sequence 
$(B_n/n)_{n\geq1}$.  Indeed, if we knew that there was a \emph{unique}
probability measure $\mu$ on $[0,\oo)$ whose moments are the sequence
$(m_k)_{k\geq1}$, then weak convergence would follow from standard
compactness arguments.  The most famous condition for uniqueness is
\emph{Carleman's criterion}, which states that $\mu$ is unique
provided $\sum_{k\geq1} m_k^{-1/2k}=\oo$.  In our case, however, this
condition is not satisfied.
\end{remark}

Our second main result says that the degree distribution is the same as
for random recursive trees (without doubling):
\begin{theorem}[Degree distribution]\label{th:degree_prof}
For all $i, n\geq 0$,
let $U_i(n)$ be the number of nodes in~$\tau_n$ that have exactly~$i$ children.
Almost surely as $n\uparrow\infty$,
\[\frac{U_i(n)}{|\tau_n|} \to \frac1{2^{i+1}}.\]
\end{theorem}
The equivalent result for the random recursive tree is due 
to Mahmoud and Smythe~\cite{MS92} (see also \cite{Janson05}), 
who also prove that the fluctuations are Gaussian.
The fact that the asymptotic degree distribution is the same as in the random
recursive trees can be expected from the observation that, except at doubling times, 
the tree does grow like a random recursive tree, while at doubling
times the degree distribution stays roughly unchanged. 
In the case of the random recursive tree, one can use standard results
for P\'olya urns since, for all $m\geq 1$, the vector 
$(U_1(n), \ldots, U_{m-1}(n), \sum_{i\geq m}U_i(n))$ is a P\'olya urn.
 In our case, because of the doubling steps, we no longer have a
 P\'olya urn. Instead we use stochastic approximation methods, which are
 also classical in the context of urns; 
 see \cite{Pemantle} for a survey on stochastic approximation, 
 and~\cite{Benaim, Duflo} for books on the topic.

Our third main result concerns what is called the (height) profile,
i.e.\ the joint distribution of the heights of uniformly random nodes:
\begin{theorem}[Height profile]\label{th:profile}
For all $k, n\geq 1$, given $\tau_n$,  let $u^{\sss (1)}_n, \ldots,
u^{\sss (k)}_n$ be $k$ nodes taken uniformly, independently at random
in $\tau_n$. 
Then, in distribution as $n\uparrow\infty$,
\[\left(\frac{|u^{\sss(1)}_n|-  \frac{2\log n}{1+\log 2}}{\sqrt{\frac{\log n}{1+\log 2}}}, \ldots,
\frac{|u^{\sss(k)}_n|-  \frac{2\log n}{1+\log 2}}{\sqrt{\frac{\log n}{1+\log 2}}}\right)
\Rightarrow (V + W_1, \ldots, V+W_k),
\]
where $V$ is an almost surely finite random variable, and $W_1,
\ldots, W_k$ are i.i.d.\ standard Gaussian, independent of $V$. 
\end{theorem}

It is interesting to compare this result to its equivalent for the random recursive tree: 
in the case of the random recursive tree, it is known that
\[\left(\frac{|u^{\sss(1)}_n|-  \log n}{\sqrt{\log n}}, \ldots,
\frac{|u^{\sss(k)}_n|-  \log n}{\sqrt{\log n}}\right)
\Rightarrow (W_1, \ldots, W_k).\]
(See Devroye~\cite{Devroye88} and Dobrow~\cite{Dobrow96}  for convergence of the marginals, and~\cite{MUB} for the joint convergence.)
Perhaps as expected, the height of a typical node in the doubling tree is larger than in the random recursive tree ($\frac2{1+\log2}\log n>\log n$). 
Interestingly, the doubling events add some dependencies between the height of i.i.d.\ nodes (these dependencies are expressed in the random variable $V$ in the limit).

Note that the height profile of random trees is the object of interest of a large amount of literature: 
see, e.g., Drmota and Gittenberger~\cite{DrmotaGittenberger97} for the Catalan tree,
Chauvin, Drmota and Jabbour-Hattab~\cite{CDJH01} and Chauvin, Klein, Marckert and Rouault~\cite{CKMR05} for the binary search tree,
Schopp~\cite{Schopp10} for the $m$-ary increasing tree,
Katona~\cite{Katona05} and Sulzbach~\cite{Sulzbach08} for the preferential attachment tree, and the very recent universal result of Kabluchko, Marynych, and Sulzbach~\cite{KMS17}.
All of these papers use a martingale method that dates back to Biggins~\cite{Biggins77} in the context of branching random walks; as far as we know, this method does not apply to our setting because the doubling events remove the branching property that is crucial to this approach.
Also, our result is, as far as we know, 
the only one to show some dependence between the marginals in the limit: 
we will see in the proof that the dependent term $V$ does come directly from the doubling events, 
which dramatically impact the shape, and thus the height profile, of the whole tree.

To supplement our result on the height profile, we also prove the following lower
bound on the height $H_n$ of the tree $\tau_n$ itself, i.e.\ the
maximal height of a node:
\begin{proposition}[Lower bound on the height]\label{prop:height_LB}
Let $H_n$ denote the height of $\tau_n$. Almost surely as $n\to+\infty$,
\[H_n\geq \frac{1+\mathrm e}{1+\log 2}\cdot \log n + o(\log n).\]
\end{proposition}
Note that this lower bound is strictly  larger than 
$\frac2{1+\log2}\log n$, the order  of the height of a typical node
as given in Theorem \ref{th:profile}, which is as expected.
Again, it is interesting to compare this result to the equivalent in the case of the random recursive tree, which is due to Pittel~\cite{Pittel}: in the case of the random recursive tree, $H_n/\log n \to \mathrm e$ almost surely as $n\uparrow\infty$. 
Because $(1+\mathrm e)/(1+\log 2)<\mathrm e$, our lower bound does not allow for any definite comparison between the height of the doubling tree and that of the random recursive tree. We leave this as an open problem.

\medskip
The rest of the paper is organised as follows: in Section~\ref{sec:size}, we prove Proposition~\ref{prop:cvBn} as well as some asymptotic results on the times at which doubling events happen, which are used in the rest of the paper.
We prove convergence of the degree distribution (i.e.\ Theorem~\ref{th:degree_prof}) in Section~\ref{sec:degrees}, convergence of the height profile (Theorem~\ref{th:profile}) and the lower bound on the height of the tree in Section~\ref{sec:heights}.
Finally, Section~\ref{sec:more_doubling}, we look at the original model of Bodini in which doubling events happen at all nodes and not only at the root, and prove that, in expectation, the size of the tree is superlinear.

\subsection*{Acknowledgement}

The research of JEB was supported by Vetenskapsrådet,
grant 2023-05002, and by \emph{Ruth och Nils Erik Stenbäcks stiftelse}.
CM would like to thank Jason Schweinsberg for preliminary discussions 
on the asymptotic behaviour of $B_n$ as $n\uparrow\infty$.
We would like to thank Lazare Le Borgne 
for spotting and correcting a mistake in the proof of 
Theorem~\ref{th:degree_prof},
and the anonymous referee for reading our paper
carefully and making several very helpful comments and suggestions.

\section{Asymptotic analysis of the number of nodes and the doubling times}\label{sec:size}

In this section, we prove Proposition~\ref{prop:cvBn}, and state and prove a number of preliminary results
which will subsequently be used in the proofs of our other main results.

\subsection{Asymptotics of the number of nodes}
The aim of this section is to prove Proposition~\ref{prop:cvBn}.
We start with the following lemma:
\begin{lemma}\label{lem:trick_BC}
Almost surely, $\sum_{n\geq 0}\frac1{B_n} = \infty$.
\end{lemma}

\begin{proof}
This follows from L\'evy's extension of the Borel--Cantelli lemma (see, e.g.\
\cite[12.15]{Williams}). 
Let $\c F_n$ denote the $\sigma$-algebra generated by 
$\tau_0,\dotsc,\tau_n$ and
 let $\mathcal D_n$ be the event that
$\nu_n=\varnothing$, i.e.\
at time~$n$ we pick the root of the tree. 
By definition of the model, for all $n\geq 0$, 
\[
\mathbb P(\mathcal D_{n}|\mathcal F_n) = 
\frac1{|\tau_n|}=\frac1{B_n+1}.
\]
On the event that $\sum_{n\geq 0}\frac1{B_n} < \infty$, we have that, almost surely,
\[
\sum_{n\geq 0} \mathbb P(\mathcal D_{n}|\mathcal F_n) =
\sum_{n\geq 0} \frac1{B_n+1}\leq \sum_{n\geq 0} \frac1{B_n} <\infty,
\]
which implies that, almost surely, there exists $n_0$ such that, for
all $n\geq n_0$, $\mathcal D_n$ does not occur. 
Then
\[
\sum_{n\geq 0} \mathbb P(\mathcal D_{n}|\mathcal F_n) 
\geq \sum_{n\geq n_0} \frac1{B_n+1} = \sum_{n\geq n_0} \frac1{B_{n_0}+n-n_0}
= \infty.
\]
This contradiction means that $\PP(\sum_{n\geq 0}\frac1{B_n} < \infty)=0$,
as required.
\end{proof}

\begin{remark}
In fact, we will prove later that
$\frac1{\log n}\sum_{i= 0}^{n-1}\frac1{B_i} \to \frac1{1+\log2}$, 
in probability as $n\to\oo$
(see \eqref{eq:equiv_sum_B})
and in fact almost surely (see Remark \ref{rk:as-conv}).
\end{remark}

\begin{proof}[Proof of Proposition \ref{prop:cvBn}]
We start with \eqref{eq:cv_moments};  as we explain below, 
the remaining claims  are
simple consequences of the convergence in \eqref{eq:cv_moments}.
We proceed by induction on $k\geq1$.
First note that, if we let $\mathcal F_n = \sigma(B_0, \ldots, B_n)$, then
\[
\mathbb E[B_{n+1}|\mathcal F_n] 
= \frac1{B_n+1} \cdot (2B_n+2) + \frac{B_n}{B_n+1}\cdot(B_n+1) =
B_n+2.
\]
Indeed, with probability $1/(B_n+1)$ a doubling occurs, in which case
$B_{n+1} = 2B_n+2$, while
with probability $1/(B_n+1)$, we just add one non-root node,
i.e.\ $B_{n+1} = B_n+1$. 
Thus $\mathbb E[B_{n+1}] = \mathbb E[B_n] + 2$, which implies 
$\mathbb E[B_n] = 2n$ for all $n\geq 0$.
This concludes the proof of \eqref{eq:cv_moments} in the base case $k=1$.

For the induction step, we assume that~\eqref{eq:cv_moments} holds for
all $\ell<k$. 
Now note that, for all $n\geq 0$,
\ba
\mathbb E[B_{n+1}^k\mid\mathcal F_n] 
&= \frac1{B_n+1}\cdot (2B_n+2)^{k} + \frac{B_n}{B_n+1}\cdot (B_n+1)^{k}
= \big(B_n + 2^k\big)(B_n+1)^{k-1}\\
&= \big(B_n + 2^k\big) \sum_{\ell = 0}^{k-1} \binom{k-1}\ell B_n^\ell
= B_n^k +\sum_{\ell = 1}^{k-1} \bigg[\binom{k-1}{\ell-1} + 2^k\binom{k-1}\ell \bigg] B_n^\ell + 2^k,
\ea
which implies
\[\mathbb E[B_{n+1}^k]
= \mathbb E[B_n^k] +\sum_{\ell = 1}^{k-1} \bigg[\binom{k-1}{\ell-1} + 2^k\binom{k-1}\ell \bigg] \mathbb E[B_n^\ell] + 2^k.\]
This implies that, for all $n\geq 0$,
\[\mathbb E[B_n^k] 
= 1 + \sum_{\ell = 1}^{k-1} \bigg[\binom{k-1}{\ell-1} + 2^k\binom{k-1}\ell \bigg] 
\sum_{i=0}^n \mathbb E[B_i^\ell] + (n+1) 2^k.\]
By the induction hypothesis, for all $\ell<k$,
$\mathbb E[B_n^\ell] = (m_\ell+o(1)) n^{\ell}$ as $n\uparrow\infty$.
Thus, for all $\ell<k$, we can write
$\mathbb E[B_i^\ell] = (m_\ell+\eps(i,\ell)) i^{\ell}$ where
$\eps(i,\ell)\to0$ as $i\to\oo$, and obtain
\[
  \sum_{i=0}^n \mathbb E[B_i^\ell] =
\sum_{i=0}^n i^\ell(m_\ell+\eps(i,\ell))=
m_\ell\frac{n^{\ell+1}(1+o(1))}{\ell+1}+
n^{\ell+1}\frac1n\sum_{i=0}^n\big(\tfrac i n\big)^\ell\eps(i,\ell)
=  \frac{m_\ell +o(1)}{\ell+1} \cdot n^{\ell+1},
\]
since $\frac1n\sum_{i=0}^n\big(\frac i n\big)^\ell\eps(i,\ell)\to0$.
In total, we thus get that
\[
\mathbb E[B_n^k] = \bigg(\frac{(k-1+ 2^k)m_{k-1}}{k}+o(1)\bigg)n^k.\]
Using the expression
$m_{k-1} = 2\prod_{i=2}^{k-1} \big(1-\frac1i+\frac{2^i}{i}\big)$, we get
\[
\mathbb E[B_n^k] \sim 2 n^k\prod_{i=2}^{k} \bigg(1-\frac1i+\frac{2^i}{i}\bigg),
\]
as claimed, which concludes the proof of~\eqref{eq:cv_moments}.

Tightness follows from the fact that $\mathbb E[B_n/n] \to 2$: indeed, by Markov's inequality, for any $K>0$,
\[\sup_{n\geq 0}\mathbb P(B_n/n\geq K) 
\leq \frac{\sup_{n\geq 0}\mathbb E[B_n/n]}K.\]
Because $\mathbb E[B_n/n] \to 2$, we have that $\sup_{n\geq 0}\mathbb E[B_n/n]<\infty$, and thus, for all $\varepsilon>0$, there exists $K = K(\varepsilon)$ such that $\sup_{n\geq 0}\mathbb P(B_n/n\geq K) \leq \varepsilon$, as desired.

Now let $(\omega(n))_{n\geq 0}$ be a sequence such that $n = o(\omega(n))$ as $n\uparrow\infty$;
for all $\varepsilon>0$,
\[\mathbb P\bigg(\frac{B_n}{\omega(n)}>\varepsilon\bigg)
= \mathbb P\bigg(\frac{B_n}{n}>\frac{\varepsilon \omega(n)}{n}\bigg)
\leq \frac{\mathbb E[B_n/n]}{\varepsilon \omega(n)/n}
\to 0,\]
as $n\uparrow\infty$, as claimed.
Finally, fix $\eta>0$ and choose an integer $k$ such that $k\eta>1$.
For all $\varepsilon>0$, by Markov's inequality,
\[\mathbb P\bigg(\frac{B_n}{n^{1+\eta}}>\varepsilon\bigg)
\leq \frac{\mathbb E[(B_n/n)^k]}{\varepsilon^k n^{k\eta}},\]
which is summable because $(\mathbb E[(B_n/n)^k])_{n\geq 0}$ is
convergent and thus bounded. 
By the first Borel--Cantelli lemma, this implies that $B_n/n^{1+\eta}$
converges almost surely to~0, as claimed. 
\end{proof}

\subsection{Asymptotics of doubling times}

We now consider the number of doubling events
before time~$n$, i.e.\ the random variable
\be\label{eq:kappa_def}
\kappa(n) := \sum_{i=1}^n {\bf 1}_{\nu_{i-1} = \varnothing}.
\ee
The following results will be useful in the proof of Theorem
\ref{th:profile}.

\begin{proposition}\label{prop:last_doubling}
As $n\uparrow\infty$,
\[
\frac{\kappa(n)-\frac{\log n}{1+\log 2}}
{\frac{\sqrt{\log n}}{(1+\log2)^{3/2}}}
\Rightarrow \mathcal{N}(0,1).
\]
\end{proposition}

\begin{remark}
Using the continuous-time embedding in 
Section \ref{sec:height_LB}
 (see Remark \ref{rk:as-conv})
one can also show that
\[
\frac{\kappa(n)}{\log n}\to \frac1{1+\log2},\qquad
\text{ almost surely as }n\to\oo.
\]
\end{remark}

We define
the sequence $(s_n)_{n\geq 0}$ of doubling times as follows: 
 $s_0 = 0$ and, for all $n\geq 0$,
\[
s_{n+1} = \min\{k>s_n\colon \nu_{k-1} = \varnothing\}.
\]
To prove Proposition \ref{prop:last_doubling}, we start by looking at 
the sequence $(s_n, B_{s_n})_{n\geq 0}$. 
To simplify notation, we set $C_n = B_{s_n}$ and
we set  $\Delta s_{n+1} = s_{n+1}-s_n$, for all $n\geq 0$.
Note that, by definition, for all $n\geq 0$, for all $x\in\{0, 1, 2, \ldots\}$,
\[\mathbb P(\Delta s_{n+1} > x | s_n, C_n)
= \frac{C_n}{C_n+1}\cdot \frac{C_n+1}{C_n+2}\cdots \frac{C_n+x-1}{C_n+x}
= \frac{C_n}{C_n+x}.\]
Equivalently,
\[\mathbb P(\Delta s_{n+1} \leq x | s_n, C_n)
= \frac{x}{C_n+x}.\]
Thus, if $(U_n)_{n\geq 1}$ is a sequence of i.i.d.\ uniform random
variables on $[0,1]$, then for all $n\geq 0$, 
\be\label{eq:rec_s}
\Delta s_{n+1} 
 \eqd 
\bigg\lceil \frac{U_{n+1}C_n}{1-U_{n+1}}\bigg\rceil,
\ee
where $\eqd$ means equality in distribution.
Furthermore, because, at time $s_{n+1}-1$, the number of non-root nodes in the tree is $C_n+\Delta s_{n+1}-1$, we have
\be\label{eq:rec_C}
C_{n+1} = 2(C_n+\Delta s_{n+1}-1)+2
= 2(C_n+\Delta s_{n+1})
 \eqd 
 2\bigg(C_n+\bigg\lceil \frac{U_{n+1}C_n}{1-U_{n+1}}\bigg\rceil\bigg).
\ee
 In what follows, we treat the distributional equalities
  in \eqref{eq:rec_s} and \eqref{eq:rec_C} as actual equalities, in
  other words we replace the random variables $s_n$ and $C_n$ with
  distributional copies satisfying these equalities.

\begin{lemma} \label{lem:cv_sn}
In distribution as $n\uparrow\infty$,
\[
\Big(\frac{\log C_n - (1+\log 2)n}{\sqrt n},
\frac{\log s_n - (1+\log 2)n}{\sqrt n}
\Big)
\Rightarrow (N,N),
\]
where $N\sim \mathcal N(0,1)$ is a standard normal random variable. 
\end{lemma}

\begin{proof}
First note that, by~\eqref{eq:rec_C}, for all $n\geq 0$,
\[C_{n+1} \geq 2\bigg(C_n+ \frac{U_{n+1}C_n}{1-U_{n+1}}\bigg)
= \frac{2C_n}{1-U_{n+1}}.\]
By induction, we thus get
\be\label{eq:LB_C}
C_n \geq 2^{n-1} C_1 \prod_{i=2}^n \frac1{1-U_i}
= 2^n \prod_{i=2}^n \frac1{1-U_i},
\ee
because, by definition, $s_1 = 1$ and $C_1 = B_{s_1} = 2$.
On the other hand, \eqref{eq:rec_C} also implies that, for all $n\geq 0$
\[C_{n+1}\leq 2\bigg(C_n+ \frac{U_{n+1}C_n}{1-U_{n+1}}+1\bigg)
= 2\bigg(\frac{C_n}{1-U_{n+1}} + 1\bigg).\]
By induction, we thus get
\ban
C_n 
&\leq  2^{n-1} C_1 \prod_{i=2}^n \frac1{1-U_i} 
+\sum_{k=3}^{n+1} 2^{n+2-k} \prod_{i=k}^{n} \frac1{1-U_i} 
= 2^{n} \prod_{i=2}^n \frac1{1-U_i} 
+\sum_{k=3}^{n+1} 2^{n+2-k} \prod_{i=k}^{n} \frac1{1-U_i} \notag\\
&\leq \bigg(2^{n} +\sum_{k=3}^{n+1} 2^{n+2-k}\bigg)\prod_{i=2}^n \frac1{1-U_i} 
= 2^n \bigg(1+\sum_{k=1}^{n-1} (\nicefrac12)^{k}\bigg)\prod_{i=2}^n \frac1{1-U_i} 
\leq 2^n \bigg(1+\sum_{k\geq 1} (\nicefrac12)^{k}\bigg)\prod_{i=2}^n \frac1{1-U_i} \notag\\
&= 2^{n+1} \prod_{i=2}^n \frac1{1-U_i}.\label{eq:UB_C}
\ean
In total, we have thus proved that
\[2^n \prod_{i=2}^n \frac1{1-U_i}\leq C_n\leq 2^{n+1} \prod_{i=2}^n \frac1{1-U_i},\]
which implies that, as $n\uparrow\infty$,
\be\label{eq:clt1}
\log C_n 
= n\log 2 + \sum_{i=2}^n \log\Big(\frac1{1-U_i}\Big) + \mathcal O(1).
\ee
Next,
by~\eqref{eq:rec_s}, for all $n\geq 0$,
\[s_{n+1} \geq s_n + \frac{U_{n+1}C_n}{1-U_{n+1}},
\]
which implies, by induction,
\[s_n \geq s_1 + \sum_{k=2}^n \frac{U_kC_{k-1}}{1-U_k} 
\geq \frac{U_nC_{n-1}}{1-U_n}\geq 2^{n-1} U_n\prod_{i=2}^n \frac1{1-U_i},\]
where we have used~\eqref{eq:LB_C}.
On the other hand, using~\eqref{eq:rec_s} again 
we get that, for all $n\geq 0$,
\[
s_{n+1}\leq s_n + \frac{U_{n+1}C_n}{1-U_{n+1}}+1,
\]
which implies, using induction and \eqref{eq:UB_C}, that
\[
s_n\leq n + \sum_{k=2}^n  \frac{U_kC_{k-1}}{1-U_k}
\leq n + \sum_{k=2}^n 2^kU_k\prod_{i=2}^{k} \frac1{1-U_i}.
\]
Now, because $U_i\in (0,1)$ almost surely for all $i\geq 1$, we get
\[s_n\leq n + 2^n \bigg(\prod_{i=2}^n \frac1{1-U_i}\bigg) \sum_{k=0}^{n-2} (\nicefrac12)^{k}
\leq n + 2^n \bigg(\prod_{i=2}^n \frac1{1-U_i}\bigg) \sum_{k\geq 0} (\nicefrac12)^{k}
= n+2^{n+1} \prod_{i=2}^n \frac1{1-U_i}.\]
In total, we have thus proved that for all $n\geq 1$,
\[
2^{n-1} U_n\prod_{i=2}^n \frac1{1-U_i}\leq s_n
\leq n + 2^{n+1} \prod_{i=2}^n \frac1{1-U_i},
\]
which implies that
\be\label{eq:clt2}
\log s_n = n\log 2 + \sum_{i=2}^n \log\Big(\frac1{1-U_i}\Big) 
+ \mathcal O(\log n).
\ee
Applying the central limit theorem applied to the 
sequence of i.i.d.\ random variables $(\log(1/(1-U_i))_{i\geq 1}$,  
which have expectation and variance both equal to~$1$, 
the result follows from 
\eqref{eq:clt1}
and
\eqref{eq:clt2}.
\end{proof}

\begin{proof}[Proof of Proposition \ref{prop:last_doubling}]
The argument is inspired by standard arguments in renewal
theory, with $s_k$ playing the role of the time of the $k$'th renewal
and $\kappa(n)$ the number of renewals up to time~$n$.
By definition, for any $k\in\bb N$,
we have that $\kappa(n)\geq k$ if and only if $s_k\leq n$.
Now, for any $x\in\mathbb{R}$,
\[
\PP\bigg(\frac{\kappa(n)-\frac{\log n}{1+\log2}}
{\frac{\sqrt{\log n}}{(1+\log 2)^{3/2}}}\geq -x\bigg)=
\PP(\kappa(n)\geq k_n(x))=\PP(s_{k_n(x)}\leq n)
\]
where
\[\textstyle
k_n(x)=\Big\lceil
\frac{\log n}{1+\log 2}-\frac{x\sqrt{\log n}}{(1+\log 2)^{3/2}}
\Big\rceil.
\]
But
\[
\PP(s_{k_n(x)}\leq n)=
\PP\bigg(\frac{\log s_{k_n(x)}-(1+\log2)k_n(x)}
{\sqrt{k_n(x)}}\leq
\frac{\log n-(1+\log2)k_n(x)}
{\sqrt{k_n(x)}}
\bigg)
\]
and as $n\to\oo$ we have $k_n(x)\to\oo$
and
\[
\frac{\log n-(1+\log2)k_n(x)}
{\sqrt{k_n(x)}}\to x.
\]
Thus by Lemma \ref{lem:cv_sn},
with $\Phi$ the cumulative density function of the standard normal
distribution, 
\[
\PP(s_{k_n(x)}\leq n)\to \Phi(x)=1-\Phi(-x),
\]
as required.
\end{proof}

\section{The degree distribution: 
proof of Theorem \ref{th:degree_prof}}\label{sec:degrees}

The proof of Theorem \ref{th:degree_prof} is based on stochastic
approximation, specifically 
 the following result, attributed to Robbins and Siegmund~\cite{RS71}:
\begin{theorem}[{e.g.\ \cite[Theorem 1.3.12]{Duflo}}]\label{th:RS}
Suppose that $(V(n))_{n\geq 0}$, $(\alpha_n)_{n\geq 0}$, $(\beta_n)_{n\geq 0}$, and $(\gamma_n)_{n\geq 0}$ are four non-negative sequences adapted to a filtration $(\mathcal F_n)_{n\geq 0}$ and satisfying, for all $n\geq 0$,
\[\mathbb E[V(n+1)|\mathcal F_n]\leq (1+\alpha_n)V(n) - \beta_n + \gamma_n.\]
Then, on the event that $\sum_{n\geq 0}\alpha_n<\infty$ and $\sum_{n\geq 0}\gamma_n<\infty$,
we have that, almost surely, $(V(n))_{n\geq 0}$ converges to a finite
random variable, and $\sum_{n\geq 0} \beta_n<\infty$. 
\end{theorem}

\begin{proof}[Proof of Theorem~\ref{th:degree_prof}]
Recall that $U_i(n)$ denotes the number of nodes in $\tau_n$ that have
exactly $i$ children.
We fix $m\geq 1$ and let, for all $0\leq i\leq m$,
\[X_i(n) = \begin{cases}
U_i(n) & \text{ if } 0\leq i\leq m-1\\
\sum_{j\geq m} U_j(n) & \text{ if }i=m,
\end{cases}\]
and set 
\[\hat X_i(n) = \frac{X_i(n)}{B_n+1}.\]
By definition, for all $0\leq i\leq m-1$, 
$\hat X_i(n)$ is the proportion of nodes having $i$ children in
$\tau_n$.
Let us set $\Delta X_i(n+1) = X_i(n+1) - X_i(n)$ and 
$\Delta B_{n+1} = B_{n+1} - B_n$.
For all $0\leq i\leq m$, for all $n\geq 0$, 
\ba
\hat X_i(n+1) 
&= \frac{X_i(n)+\Delta X_i(n+1)}{B_{n+1}+1}
= \hat X_i(n) \cdot \frac{B_n+1}{B_{n+1}+1} + \frac{\Delta X_i(n+1)}{B_{n+1}+1}\\
&= \hat X_i(n) + \frac1{B_{n+1}+1} \big(\Delta X_i(n+1) - 
\Delta B_{n+1}\hat X_i(n)\big).
\ea
Note that, by definition of the model, 
with probability $1/(B_n+1)$, we pick the root and 
double the number of nodes
with $i$ children (for all $i\geq1$) 
and add one node with two children, whilst, with
probability $B_n/(B_n+1)$, we pick a non-root 
node uniformly at random and increase its number of children by one
 (in this case, with probability $\hat X_i(n)$, the number of nodes
 with $i$ children decreases by one, and with probability 
$\hat  X_{i-1}(n)$, it increases by one). 
Hence,
for all $1\leq i\leq m-1$,
\ba\mathbb E\big[\Delta X_i(n+1)|\mathcal F_n\big]
&= \frac1{B_n+1} \cdot (X_i(n)+{\bf 1}_{i=2}) + \frac{B_n}{B_n+1}\cdot (\hat X_{i-1}(n)-\hat X_i(n))\\
&= \hat X_{i-1}(n) + \frac{{\bf 1}_{i=2}+\hat X_i(n)-\hat X_{i-1}(n)}{B_n+1}
\ea
Similarly,
in the case when the tree does not double, 
the number of leaves (nodes with 0 children) always increases by one,
except if the node we have picked was itself a leaf.  Hence
\[
\mathbb E\big[\Delta X_0(n+1)|\mathcal F_n\big]
= \frac1{B_n+1} \cdot X_0(n) + \frac{B_n}{B_n+1}\cdot (1-\hat X_0(n))
= 1- \frac{1-\hat X_0(n)}{B_n+1}.
\]
Because $\sum_{i=0}^m \hat X_i(n) = 1$, we can write
\[\mathbb E\big[\Delta X_0(n+1)|\mathcal F_n\big]
= \sum_{i=0}^m \hat X_i(n) - \frac{1-\hat X_0(n)}{B_n+1}.\]
Finally,
\[\mathbb E\big[\Delta X_m(n+1)|\mathcal F_n\big]
= \frac1{B_n+1} \cdot X_m(n) + \frac{B_n}{B_n+1}\cdot \hat X_{m-1}(n)
= \hat X_m(n) + \hat X_{m-1}(n) - \frac{\hat X_{m-1}(n)}{B_n+1}.\]
Note that,  also,
\[
\mathbb E\big[\Delta B_{n+1}|\mathcal F_n\big]
= \frac1{B_n+1} \cdot (B_n+2) + \frac{B_n}{B_n+1}
= 2.
\]
Introduce, for all $0\leq i\leq m$,
\[\Delta M_i(n+1) 
= \Delta X_i(n+1) - \Delta B_{n+1}\hat X_i(n) - \mathbb E\big[\Delta
X_i(n+1) - \Delta B_{n+1}\hat X_i(n)|\mathcal F_n\big],
\]
and set, for all $x = (x_0, \ldots, x_m)\in \mathbb R^{m+1}$,
\[
F_i(x) = \begin{cases}
\sum_{i=1}^m x_i - x_0  & \text{ if }i=0,\\
x_{i-1}- 2x_i & \text{ if }1\leq i\leq m-1,\\
x_{m-1} - x_m & \text{ if } i=m.
\end{cases}
\]
Also let
\be\label{eq:def_eps}
\varepsilon_i(n+1) = \frac{1}{B_n+1}
\begin{cases}
- (1- \hat X_0(n)) & \text{ if }i=0,\\
{\bf 1}_{i=2}+\hat X_i(n)-\hat X_{i-1}(n) &\text{ if }1\leq i\leq m-1,\\
- \hat X_{m-1}(n) &\text{ if }i=m.
\end{cases}
\ee
Using the above, we can write
\begin{equation}\label{eq:SA1D}
\hat X_i(n+1) = \hat X_i(n) + \frac1{B_{n+1}+1} \big(F_i(\hat X(n))+\Delta M_i(n+1)+\varepsilon_i(n+1)\big),
\end{equation}
We write~\eqref{eq:SA1D} as an identity on vectors:
\be\label{eq:SA}
\hat X(n+1) = \hat X(n)+\frac1{B_{n+1}+1} \big(F(\hat X(n))+\Delta M(n+1)+\varepsilon(n+1)\big).
\ee
Now, because $B_{n+1}$ is not $\mathcal F_n$-measurable, we re-write this as
\ba
\hat X(n+1) =  \hat X(n)
&+\frac1{B_n+1} \big(F(\hat X(n))+\Delta M(n+1)+\varepsilon(n+1)\big)\\
&- \frac{B_{n+1}-B_n}{(B_n+1)(B_{n+1}+1)}\cdot Y(n+1),
\ea
where we have set 
\be\label{eq:def_Y}
Y(n+1) = F(\hat X(n))+\Delta M(n+1)+\varepsilon(n+1).\ee
In total, this gives
\be\label{eq:SA2}
\hat X(n+1) =  \hat X(n)+\frac1{B_n+1} \big(F(\hat X(n))+\Delta M(n+1)+\eta(n+1)\big),
\ee
where
\be\label{eq:def_eta}
\eta(n+1)= \varepsilon(n+1)-\frac{\Delta B_{n+1}}{B_{n+1}+1} \cdot Y(n+1).
\ee
This recursion is of the form of a stochastic approximation. 
However, the step sizes $(1/(B_n+1))_{n\geq 0}$ are random, and we have a random error term $(\eta(n+1))_{n\geq 0}$. Because of these two reasons, we cannot apply a theorem directly from the literature, but need instead to write a specific argument.
We now let $v_i = 2^{-i-1}$ for all $0\leq i\leq m-1$ and $v_m = 2^{-m}$.
One can check that $F(v) = 0$; in fact, for all $x\in\mathbb R^d$, $F(x) = Ax$, with
\[A = \begin{pmatrix}
-1 & 1 & 1  & \ldots & \ldots & 1\\
1 & -2 & 0 & \ldots & \ldots & 0\\
0 & 1 & \ddots & \ddots & & \vdots \\
\vdots & \ddots&\ddots & \ddots & \ddots & \vdots \\
\vdots &  & \ddots & 1 & -2 & 0 \\
0 & \ldots & \ldots &  0 &   1 & -1\\
\end{pmatrix}\]
and one can check that the largest eigenvalue of $A$ is 0, it is a simple eigenvalue with eigenvector $v$, the unique one with non-negative coefficients and satisfying $\sum_{i=0}^m v_i = 1$.
We thus have (write $\|\cdot\|$ for the $L^2$ norm on $\mathbb Z^{m+1}$), for all $n\geq 0$,
\ban
\|\hat X(n+1) - v\|^2
=  \|\hat X(n)-v\|^2 + \frac2{B_n+1} \langle \hat X(n)-v, A\hat X(n) + \Delta M(n+1) + \eta(n+1)\rangle &\notag\\
+ \frac1{(B_n+1)^2} \|A\hat X(n) + \Delta M(n+1) + \eta(n+1)\|^2.&\label{eq:ineq_norm1}
\ean
We use the triangle inequality and the fact that $(x+y)^2\leq 2x^2 + 2y^2$
for all $x,y\in\mathbb R$, to get
\ban
\|A\hat X(n) + \Delta M(n+1) + \eta(n+1)\|^2
&\leq 2\|A\hat X(n)\|^2 + 2 \|\Delta M(n+1) + \eta(n+1)\|^2\notag\\
&\leq 2\|A(\hat X(n)-v)\|^2 + 4 \|\Delta M(n+1)\| + 4\|\eta(n+1)\|^2\notag\\
&\leq 2\vvvert A\vvvert^2 \|\hat X(n)-v\|^2 + 4 \|\Delta M(n+1)\|^2 + 4\|\eta(n+1)\|^2,
\label{eq:ineq_norm2}
\ean
where $\vvvert A\vvvert$ is the operator norm of $A$.
Before proceeding, we show that there exists a constant $C>0$ such that
\be\label{eq:sups_finite}
\sup_{n\geq 0} \|\Delta M(n+1)\|\leq C\quad \text{ and }\quad
\sup_{n\geq 0} \|\eta(n+1)\|\leq C.
\ee
Indeed, first recall that
\[\Delta M(n+1) = \Delta X(n+1)-\Delta B_{n+1}\hat X(n) -
\mathbb E[\Delta X(n+1)-\Delta B_{n+1}\hat X(n)|\mathcal F_n].\]
On the event that the tree doubles at time $n+1$, we have 
(with  $e_2 = (0, 0 ,1 , 0, \ldots, 0)^t$; to be consistent with $X(n)
= (X_0(n), \ldots, X_m(n))^t$, we let $e_0 = (1, 0, 0, \ldots, 0)^t$,
$e_1 = (0, 1, 0, \ldots, 0)^t$, etc) 
\ba
\Delta X(n+1)-\Delta B_{n+1}\hat X(n) 
&= X(n)+2e_2 - (B_n+2) \hat X(n)\\
&= X(n)-(B_n+1)\hat X(n) + 2e_2 - \hat X(n) 
= 2e_2 -\hat X(n),
\ea
which implies $\|\Delta X(n+1)-\Delta B_{n+1}\hat X(n)\|\leq 3$ 
(because $\|\hat X(n)\|^2\leq \sum_{i=0}^m \hat X_i(n) = 1$, as $\hat X(n)$ has non-negative coefficients that sum to~1).
On the event that the tree does not double at time $n+1$, 
\[\|\Delta X(n+1)-\Delta B_{n+1}\hat X(n)\|
= \|\Delta X(n+1) - \hat X(n)\|\leq \|\Delta X(n+1)\| + 1,\]
and $\|\Delta X(n+1)\|$ is bounded by the maximum of the norms of the columns of $A+I$, 
which is a constant, which we let~$K$ denote.
In total, we thus get that
\[\|\Delta X(n+1)-\Delta B_{n+1}\hat X(n)\|\leq K+3,
\]
for all $n\geq 0$,
which implies that 
$\sup_{n\geq 0} \|\Delta M(n+1)\|\leq 2(K+3)$.
We now prove that $\sup_{n\geq 0}\|\eta(n+1)\|<\infty$ (see~\eqref{eq:def_eta} for the definition of $\eta(n+1)$). First note that, by definition (see~\eqref{eq:def_eps} for the definition of $\varepsilon(n+1)$),
\[\|\varepsilon(n+1)\| \leq \frac4{B_{n+1}+1}\leq 4.\]
Thus, by the triangle inequality (see~\eqref{eq:def_Y} for the
definition of $Y(n+1)$), 
\[\|Y(n+1)\| \leq \|A\hat X(n)\| + \|\Delta M(n+1)\| + \|\varepsilon(n+1)\|
\leq \vvvert A\vvvert + (K+3) + 4,\]
implying that $\sup_{n\geq 0 }\|Y(n+1)\|\leq \vvvert A\vvvert +  K + 7$.
Thus, for all $n\geq 0$,
\be\label{eq:ineq_eta}
\|\eta(n+1)\|
\leq \|\varepsilon(n+1)\| + \frac{\Delta B_{n+1}}{B_{n+1}+1} \cdot \|Y(n+1)\|
\leq \|\varepsilon(n+1)\| + \|Y(n+1)\| \leq \vvvert A\vvvert +  K + 11,
\ee
which concludes the proof of~\eqref{eq:sups_finite} (we choose 
$C\geq \vvvert A\vvvert +  K + 11$). 
We now let $V(n) = \|\hat X(n) - v\|^2$ for all $n\geq 0$; 
with this notation, and using the triangle inequality,
we get from~\eqref{eq:ineq_norm1} and~\eqref{eq:ineq_norm2} that
\ban
\mathbb E[V(n+1)|\mathcal F_n]
\leq {} & V(n) + \frac2{B_n+1} \langle \hat X(n)-v, A\hat X(n)\rangle
+ \frac2{B_n+1} \langle \hat X(n)-v, \mathbb E[\eta(n+1)|\mathcal F_n]\rangle\notag\\
&+ \frac1{(B_n+1)^2} \big(\vvvert A\vvvert V(n) + 2C\big),\label{eq:tow_Duflo}
\ean
where we have chosen $C$ larger than
$\sup_{n\geq 0}\Delta M(n+1)$ and 
$\sup_{n\geq 0}\eta(n+1)$.
Now, by the Cauchy--Schwarz and Jensen inequalities,
\be\label{eq:ineq_angle}
\langle \hat X(n)-v, \mathbb E[\eta(n+1)|\mathcal F_n]\rangle
\leq \|\hat X(n)-v\| \mathbb E[\|\eta(n+1)\||\mathcal F_n]
\leq 2 \mathbb E[\|\eta(n+1)\||\mathcal F_n],
\ee
because $\|\hat X(n)-v\|\leq \|\hat X(n)\| + \|v\|\leq 2$.
Now, by~\eqref{eq:ineq_eta}, and the fact that $\varepsilon(n+1)\leq 4\leq C$ and $\|Y(n+1)\|\leq C$ for all $n\geq 0$, we have
\[
\mathbb E[\|\eta(n+1)\||\mathcal F_n]
\leq C\mathbb E\bigg[1+\frac{\Delta B_{n+1}}{B_{n+1}+1}\Big|\mathcal F_n\bigg]
= C \bigg(\frac1{B_n+1} \cdot \frac{B_n+2}{2B_n+3} + \frac{B_n}{B_n+1} \cdot \frac{1}{B_n+2}\bigg),
\]
by definition of the model.
Thus,
\[\mathbb E[\|\eta(n+1)\||\mathcal F_n]
\leq \frac C{B_n+1} \bigg(\frac{B_n+2}{2B_n+3} + \frac{B_n}{B_n+2}\bigg)\leq \frac{2C}{B_n+1}.\]
Thus, by~\eqref{eq:ineq_angle} and~\eqref{eq:tow_Duflo}, for all $n\geq 0$,
\[\mathbb E[V(n+1)|\mathcal F_n]
\leq \bigg(1+\frac{\vvvert A\vvvert}{(B_n+1)^2}\bigg)V(n)
+\frac2{B_n+1} \langle \hat X(n)-v, A\hat X(n) \rangle
+\frac{10C}{(B_n+1)^2}.\]
We want to apply Theorem~\ref{th:RS} with 
$\alpha_n = \frac{\vvvert A\vvvert}{(B_n+1)^2}$, 
$\beta_n = -\frac2{B_n+1} \langle \hat X(n)-v, A\hat X(n) \rangle$, 
and $\gamma_n = \frac{10C}{(B_n+1)^2}$, so we need to check the
conditions on these sequences. 
Because, by definition, $B_n\geq n$ for all $n\geq 0$, 
we have that, almost surely, $\sum_{n\geq 0} \alpha_n<\infty$ 
and $\sum_{n\geq 0} \gamma_n<\infty$.
To check that $\beta_n\geq 0$ first note that 
$\langle \hat  X(n)-v, A\hat X(n) \rangle = \langle \hat X(n)-v, A(\hat X(n)-v)
  \rangle$.  While the eigenvalues of $A$ are all non-positive, $A$  is
  not Hermitian and hence not negative semidefinite;  
however, one may check explicitly that
$\langle x,Ax\rangle\leq 0$ for all $x=(x_i)_{i=0}^{m}$ that satisfy
$\sum_{i=0}^m x_i=0$.  
 Indeed, for such $x$,
\begin{align*}
\langle x, Ax\rangle 
&= x_0\bigg(-x_0 + \sum_{i=1}^m x_i\bigg)
+\sum_{i=1}^{m-1} x_i(x_{i-1}-2x_i) + x_m(x_{m-1}-x_m)\\
&= -2\bigg(\sum_{i=1}^m x_i\bigg)^{\!\!2} + \sum_{i=1}^{m} x_i x_{i-1}
- 2\sum_{i=1}^{m-1} x_i^2 - x_m^2
\leq -2\sum_{i=1}^m x_i^2 - 2\sum_{i=1}^m\sum_{j=1, j\neq i}^m x_i x_j\\
&= -2\sum_{i=1}^m x_i^2 - 2\sum_{i=1}^m x_i(-x_0-x_i)
= 2x_0\sum_{i=1}^m x_i = -2x_0^2\leq 0.
\end{align*}
Applying this observation to
$x=\hat X(n)-v$ gives $\beta_n\geq 0$.
Therefore, by Theorem~\ref{th:RS}, almost surely, $W:=\lim_{n\uparrow\infty} V(n)$ exists and is finite, and $\sum_{n\geq 0}\beta_n<\infty$.
On the event that $W\neq 0$, there exists $\varepsilon>0$ such that,
for all $n$ large enough, $V(n)\geq \varepsilon$. Now note that, on
the set  $\{x\in [0,1]^{m+1}\colon \sum_{i=0}^m x_i = 1\}$
to which $\hat X(n)$ belongs for all $n\geq 0$, $x\mapsto \langle x-v,
A(x-v)\rangle$ is continuous, non negative, and its unique zero is
$v$. Thus, on $\{x\in [0,1]^{m+1}\colon \sum_{i=0}^m x_i =
1\}\cap\{\|x-v\|\geq \varepsilon\}$, the maximum of $x\mapsto \langle
x-v, A(x-v)\rangle$ is negative; we let $-c$ denote this maximum.  
We thus get that, for all $n$ large enough, $\beta_n\geq 2c/(B_n+1)$. 
By Lemma~\ref{lem:trick_BC}, this implies that $\sum_{n\geq 0} \beta_n = \infty$, which is an event of probability zero. 
Thus, $W = 0$ almost surely, i.e.\ $\lim_{n\uparrow\infty} \hat X(n) = v$ almost surely as $n\uparrow\infty$. In other words, for all $0\leq i\leq m-1$,
\[\frac{U_i(n)}{B_n+1} = \frac{X_i(n)}{B_n+1} \to \frac1{2^{i+1}}.\]
Because $m$ can be chosen arbitrarily large, this concludes the proof.
\end{proof}

\begin{remark}
Because of the step-sizes in~\eqref{eq:SA2} being random, we were unable to prove a central limit theorem for $\hat X(n)$. We leave this as an open problem.
\end{remark}

\section{The distribution of heights}\label{sec:heights}

We now turn to the height profile, Theorem \ref{th:profile}, as well
as the lower bound on the height of the tree, Proposition 
\ref{prop:height_LB}.
We will give full details for the case $k=1$
of Theorem \ref{th:profile} (the height of a typical node)
in Section \ref{sec:typical_height}. 
In Section \ref{sec:height_profile} we describe the necessary
modifications  for the case $k=2$, and give  an
outline of the case $k\geq3$.  Proposition 
\ref{prop:height_LB} is proved in Section \ref{sec:height_LB}

\subsection{The height of a typical node}\label{sec:typical_height}

Let us reformulate the case $k=1$ of  Theorem \ref{th:profile}:

\begin{proposition}\label{prop:typ_height}
For all $n\geq 0$, let $u_n$ be a uniformly  random node in $\tau_n$.
There is an a.s.\ finite random variable $\Lambda$ such that,  
as $n\uparrow\infty$,
\[
\frac{|u_n|- \frac{2\log n}{1+\log 2}}{\sqrt{\frac{\log n}
{1+\log 2}}} \Rightarrow \Lambda.
\]
\end{proposition}

For the proof of Proposition~\ref{prop:typ_height}, we define 
a process $(\tilde \tau_n, \tilde u_n)_{n\geq 0}$ such that, for all
$n\geq 0$, $(\tau_n, u_n)\eqd(\tilde \tau_n, \tilde u_n)$.
The process $(\tilde \tau_n, \tilde u_n)_{n\geq 0}$
will have the properties: (i) at non-doubling times and if $\tilde u_n$ is not the root, 
the height of $\tilde u_n$ is either unchanged or increases by 1
and (ii) at doubling-times, the
height of $\tilde u_n$ either increases by 1, or is reset to~0.
At non-doubling times and if $\tilde u_n$ is the root 
(note that then, we should not add a third child to the root), 
then $\tilde u_n$ ``jumps'' to a random node in the tree, and we have
no control over its height.
The process is an adaptation of a standard construction for the random
recursive tree (see~\cite{MUB} for a description of it in a more general case). 
In the random recursive tree case
there are no doubling times, and the height of the uniform
node is monotonically increasing 
(with increments at most~1). 
In our case, $\tilde u_n$ sometimes jumps to the root at a doubling event,
and sometimes jumps from the root to a random node in the tree at a non-doubling event.
We show that there are only finitely many of these jumping events,
so they can effectively be ignored.
In addition, we control the increase of the height of
$\tilde u_n$ due to  doublings by 
using Proposition \ref{prop:last_doubling}.

Let us now define the process $(\tilde \tau_n, \tilde u_n)_{n\geq 0}$.
We let $\tilde \tau_0 = \{\varnothing\}$ and $\tilde u_0 = \varnothing$.
Then, for all $n\geq 0$, given $(\tilde\tau_n, \tilde u_n)$, we
sample
\begin{itemize}
\item
$K(n+1)$ a Bernoulli-distributed random variable of 
parameter $1/(2|\tilde\tau_n|+1)$, and
\item $L(n+1)$, a Bernoulli-distributed random variable of 
parameter $1/(|\tilde \tau_n|+1)$.
\end{itemize}
Then $(\tilde\tau_{n+1}, \tilde u_{n+1})$ is constructed as follows: 
\begin{enumerate}
\item We let $\tilde\nu_{n}$ be a node taken uniformly at random among the nodes of $\tilde\tau_n$;
\item If $\tilde\nu_n = \varnothing$, then we define 
\[\tilde\tau_{n+1} = \{\varnothing\}\cup\{1w \colon w\in\tilde\tau_n\}\cup\{2w\colon w\in\tilde\tau_n\}.\]
Furthermore, 
\begin{itemize}
\item if $K(n+1) = 1$, then we set $\tilde u_{n+1} = \varnothing$, and 
\item if $K(n+1) = 0$, then we set $\tilde u_{n+1} = 1\tilde u_n$ or $\tilde u_{n+1} = 2\tilde u_n$ with probability $\nicefrac12$ each.
\end{itemize}
\item If $\tilde\nu_n \neq\varnothing$, then
\begin{itemize}
\item if $L(n+1) = 1$ and $\tilde u_n\neq \varnothing$, then we set $\tilde\tau_{n+1} = \tilde\tau_n\cup \{\tilde u_n i\}$ and $\tilde u_{n+1} = \tilde u_n i$, where $i = \min\{j\geq 1\colon \tilde u_n j\notin \tilde\tau_n\}$;
\item 
if $L(n+1) = 1$ and $\tilde u_n= \varnothing$, then we set
$\tilde\tau_{n+1} = \tilde\tau_n\cup \{\tilde \nu_n i\}$, where  $i =
\min\{j\geq 1\colon \tilde \nu_n j\notin \tilde\tau_n\}$ and $\tilde
u_{n+1} = \tilde \nu_n i$. 
\item if $L(n+1) = 0$, then we set $\tilde\tau_{n+1} =
  \tilde\tau_n\cup \{\tilde \nu_n i\}$, where  $i = \min\{j\geq
  1\colon \tilde \nu_n j\notin \tilde\tau_n\}$, and $\tilde u_{n+1} =
  \tilde u_n$. 
\end{itemize}
\end{enumerate}

\begin{lemma}\label{lem:coupling}
For all $n\geq 0$, $(\tilde\tau_n, \tilde u_n) \eqd (\tau_n, u_n)$.
\end{lemma}

\begin{proof}
By induction.
\end{proof}

\begin{proof}[Proof of Proposition~\ref{prop:typ_height}]
Throughout the proof, we identify $(\tau_n,u_n)$ with the
distributional copy $(\tilde\tau_n,\tilde u_n)$ and omit the tilde
from the notation.

We first let $n_0 = \max\{n\geq 0\colon L(n+1) = 1\text{ and }\tilde u_n = \varnothing\}$ and note that, by definition of $L(n+1)$, and by Lemma~\ref{lem:coupling},
\[\sum_{n\geq 0}\mathbb P(L(n+1) = 1\text{ and }\tilde u_n = \varnothing\mid\tau_n)
= \sum_{n\geq 0} \frac1{|\tau_n|+1} \cdot\frac1{|\tau_n|}<\infty,\]
almost surely because $|\tau_n|\geq n$ almost surely for all $n\geq 0$.
By Borel-Cantelli's lemma, this implies that $n_0<\infty$ almost
surely. 
Then, for each $n\geq n_0$, there are three cases:
\begin{itemize}
\item either $|u_{n+1}| = | u_n|$ (if $\nu_n \neq \varnothing$ and $L(n+1) = 0$),
\item or $| u_{n+1}| = | u_n|+1$ (if $\nu_n = \varnothing$ and 
$K(n+1)  = 0$, or if $\nu_n \neq \varnothing$ and $L(n+1) = 1$),
\item or $| u_{n+1}| = 0$ (if $\nu_n = \varnothing$ and $K(n+1) = 1$).
\end{itemize}
Let us write 
$R(n) = \max\{k\leq n\colon \nu_k = \varnothing\text{ and }K(k+1)= 1\}$
for the last time before $n$ when
$\nu_k = \varnothing$ and $K(k+1)=1$  
(we set $R(n) = 0$ if there are no such times).
We also set $r(n) = \max(R(n), n_0)$. 
Then by the above,
\ba
| u_{n}| 
&=  |u_{n_0}|{\bf 1}_{r(n) = n_0} +\sum_{i=r(n)+1}^{n} {\bf 1}_{\nu_{i-1} = \varnothing} 
+\sum_{i=r(n)+1}^{n}  {\bf 1}_{\nu_{i-1} \neq \varnothing}  L(i)\\
&=  |u_{n_0}|{\bf 1}_{r(n) = n_0} +\sum_{i=r(n)+1}^{n} {\bf 1}_{\nu_{i-1} = \varnothing} 
+\sum_{i=r(n)+1}^{n} L(i) 
- \sum_{i=r(n)+1}^{n}  {\bf 1}_{\nu_{i-1} = \varnothing}  L(i).
\ea
First note that $|u_{n_0}|{\bf 1}_{r(n) = n_0}\leq n_0 = \mathcal
O(1)$ almost surely as $n$ tends to infinity. 
Now, let $\c F_n=\sigma(\tau_0,\tau_1,\dotsc,\tau_n)$.
Then, since $|\tau_k|\geq k+1$ almost surely, for each $k\geq0$
\[
\PP(\nu_{k}=\varnothing, K(k+1)=1)=
\EE[\PP(\nu_{k}=\varnothing, K(k+1)=1\mid\c F_k)]=
\EE\Big[\frac{1}{|\tau_k|(2|\tau_k|+1)}\Big]
\leq\frac1{(k+1)(2k+3)}.
\]
It follows, by the Borel--Cantelli lemma,
that there is an a.s.\ finite random variable $R$ such that 
$R(n)\to R$ almost surely as $n\uparrow\infty$.
Thus, $r(n)\to \max(R, n_0)<\infty$ almost surely as
$n\uparrow\infty$. 
Similarly,
\[
\PP({\bf 1}_{\nu_{k} = \varnothing} L(k+1) = 1\mid\c F_k)=
\EE[\mathbb P({\bf 1}_{\nu_{k} = \varnothing} L(k+1) = 1\mid\c F_k)]
= \EE\Big[\frac1{|\tau_k|(|\tau_k|+1)}\Big]\leq 
\frac1{(k+1)(k+2)},
\]
and thus, by the Borel--Cantelli lemma,
\[\sum_{i=r(n)+1}^{n}  {\bf 1}_{\nu_{i-1} = \varnothing}  L(i)
\leq \sum_{i=1}^{n}  {\bf 1}_{\nu_{i-1} = \varnothing}  L(i) = \mathcal O(1),\]
almost surely as $n\uparrow\infty$.
Thus, almost surely as $n\uparrow\infty$,
\be\label{eq:prof_sum}
| u_n| = \sum_{i=1}^n {\bf 1}_{\nu_{i-1} = \varnothing} 
+ \sum_{i=1}^{n} L(i) 
+ \mathcal O(1).
\ee
We write \eqref{eq:prof_sum}  in the following form:
\be\label{eq:blabla}
| u_n| = 2\sum_{i=1}^n {\bf 1}_{\nu_{i-1} = \varnothing}
- \sum_{i=1}^n \bigg({\bf 1}_{\nu_{i-1} = \varnothing}-\frac1{B_{i-1}+1}\bigg)
+ \sum_{i=1}^n \bigg(L(i)-\frac1{B_{i-1}+2}\bigg)
+\c O(1).
\ee
To do this, we have used the fact that
\be\label{eq:bla}
\sum_{i=1}^n \bigg(\frac1{B_{i-1}+1} - \frac1{B_{i-1}+2}\bigg)
= \sum_{i=1}^n \frac1{(B_{i-1}+1)(B_{i-1}+2)}
\leq \sum_{i=1}^n \frac1{i^2}=\mathcal O(1).
\ee
The first summand in~\eqref{eq:blabla} is taken care of 
by Proposition~\ref{prop:last_doubling}:  as $n\uparrow\infty$,
\be\label{eq:clt_G0}
G_0(n):=\frac{\sum_{i=1}^n {\bf 1}_{\nu_{i-1} =\varnothing}
-\frac{\log n}{1+\log 2}}
{\frac{\sqrt{\log n}}{(1+\log2)^{3/2}}}
\Rightarrow G_0\sim \mathcal{N}(0,1).
\ee
We claim that, as $n\uparrow\infty$,
\be\label{eq:clt_G1}
G_1(n):=\frac{\sum_{i=1}^n 
\big({\bf 1}_{\nu_{i-1} = \varnothing} - \frac1{B_{i-1}+1}\big)}
{\sqrt{\frac{\log n}{1+\log 2}}}\Rightarrow 
G_1 \sim  \mathcal N(0,1)
\ee
and
\be\label{eq:clt_G2}
W(n):=\frac{\sum_{i=1}^n \big(L(i) -  
    \frac1{B_{i-1}+2}
\big)}
{\sqrt{\frac{\log n}{1+\log 2}}}\Rightarrow W\sim \mathcal N(0,1),
\ee
where $W$ is independent of $(G_0, G_1)$.
Taken together, \eqref{eq:clt_G0}, \eqref{eq:clt_G1} and
\eqref{eq:clt_G2} give the claim, with $\Lambda = G + W$, where
\be\label{eq:Lambda}
G:=\frac{2G_0}{1+\log 2} - G_1,
\ee
is independent of $W$.
We now proceed with the proofs of \eqref{eq:clt_G1} and
\eqref{eq:clt_G2}.

For \eqref{eq:clt_G1},  recalling that
$\mathbb E\big[{\bf 1}_{\nu_{i-1} = \varnothing}|\mathcal F_{i-1}\big]
= \frac1{|\tau_{i-1}|} = \frac1{B_{i-1}+1}$, we can write
\be\label{eq:sum_Bernoulli}
\begin{split}
\sum_{i=1}^n {\bf 1}_{\nu_{i-1} = \varnothing} 
&= \sum_{i=1}^n \mathbb E\big[{\bf 1}_{\nu_{i-1} = \varnothing}|\mathcal F_{i-1}\big]
+ \sum_{i=1}^n \big({\bf 1}_{\nu_{i-1} = \varnothing} - 
\mathbb E\big[{\bf 1}_{\nu_{i-1} = \varnothing}|\mathcal
F_{i-1}\big]\big)\\
&= \sum_{i=1}^n  \frac1{B_{i-1}+1}+ M_n,
\end{split}\ee
where 
$M_n:=\sum_{i=1}^n \big({\bf 1}_{\nu_{i-1} = \varnothing} - 
\mathbb E\big[{\bf 1}_{\nu_{i-1} = \varnothing}|
\mathcal F_{i-1}\big]\big)$ defines a martingale.  
We need to prove that 
$M_n/\sqrt{\frac{\log n}{1+\log2}}\Rightarrow\c N(0,1)$.
 The  quadratic variation of $(M_n)_{n\geq0}$ is given by
\be \label{eq:square_var}
\begin{split}
\langle M\rangle_n&=
\sum_{i=1}^n \mathbb E\big[\big({\bf 1}_{\nu_{i-1} = \varnothing} 
- \mathbb E\big[{\bf 1}_{\nu_{i-1} = \varnothing}|\mathcal F_{i-1}\big]\big)^2|\mathcal F_{i-1}\big]
=\sum_{i=1}^n\mathbb E\big[{\bf 1}_{\nu_{i-1} = \varnothing}|\mathcal F_{i-1}\big]\big(1-\mathbb E\big[{\bf 1}_{\nu_{i-1} = \varnothing}|\mathcal F_{i-1}\big]\big)\\
&= \sum_{i=1}^n \frac1{B_{i-1}+1}\bigg(1-\frac1{B_{i-1}+1}\bigg)=
\sum_{i=1}^n \frac1{B_{i-1}+1}+\c O(1),
\end{split}\ee
where we used that $B_n\geq n$ almost surely for all $n\geq 0$.
Furthermore, by Lemma~\ref{lem:trick_BC}, $\langle M\rangle_n\to+\infty$ almost surely as $n\uparrow\infty$. 
Thus, by the martingale law of large numbers \cite[12.14]{Williams}, 
$M_n = o(\langle M\rangle_n)$ almost surely as $n\uparrow\infty$. 
Using again the fact that $B_n\geq n$, we  get that 
$\langle M\rangle_n=\c O(\log n)$ almost surely as $n\uparrow\infty$
and hence
$M_n = o(\log n)$ almost surely as $n\uparrow\infty$. 
Now note that, by Proposition~\ref{prop:last_doubling}, 
\[
\sum_{i=1}^n {\bf 1}_{\nu_{i-1} = \varnothing} \sim 
\frac{\log n}{1+\log 2},
\qquad\text{ in probability.}
\]
Thus, by~\eqref{eq:sum_Bernoulli}, 
\be\label{eq:equiv_sum_B}
\begin{split}
\sum_{i=1}^n  \frac1{B_{i-1}+1}=
\sum_{i=1}^n {\bf 1}_{\nu_{i-1} = \varnothing}- M_n
&=\sum_{i=1}^n {\bf 1}_{\nu_{i-1} = \varnothing} + o(\log n) 
\qquad \text{almost surely}, \\
&\sim \frac{\log n}{1+\log 2} \qquad \text{in probability}.
\end{split}\ee
By~\eqref{eq:square_var}, this implies that 
$\langle M\rangle_n \sim \log n/(1+\log 2)$ in probability.
Thus, by the martingale central limit theorem
\cite[Thm 8.2.8]{Durrett},
\be
\frac{M_n}{\sqrt{\frac{\log n}{1+\log 2}}}\Rightarrow \mathcal N(0,1),
\ee
as required.

  For \eqref{eq:clt_G2}, 
we first reason conditionally on $\bs \nu = (\nu_k)_{k\geq 0}$ 
and thus on $(B_k)_{k\geq 0}$: 
conditionally on $\bs \nu$, $(L(i))_{i\geq 0}$ is a sequence 
of independent Bernoulli random variables of respective parameters
$1/(B_{i-1}+2)$, $i\geq 0$. 
By~\eqref{eq:bla} and~\eqref{eq:equiv_sum_B},
\[\sum_{i=1}^n \frac1{B_{i-1}+2} = \sum_{i=1}^n \frac1{B_{i-1}+1} + \mathcal O(1)
= \frac{\log n}{1+\log 2} + \mathcal O(1),
\qquad\text{in probability}.
\]
Thus, by the Lindeberg central theorem \cite[Theorem~7.2.1]{Gut}
(whose conditions are easily checked since the $L(i)$ are bounded),
 we get that, conditionally on $\bs\nu$,
\be\label{eq:clt_cond}
W(n)=\frac{\sum_{i=1}^n L(i)-\frac1{B_{i-1}+2}}{\sqrt{\frac{\log n}{1+\log 2}}}
\Rightarrow W,
\ee
where $W$ is a standard Gaussian. 
Explicitly, this means that for all continuous and bounded functions
$\varphi : \mathbb R\to\mathbb R$, 
\[
\mathbb E[\varphi(W(n))\mid\bs\nu]\to \mathbb E[\varphi(W)].
\]
Because the limit does not depend on $\bs\nu$, by dominated convergence, 
we can take expectations on both sides of the limit, which
gives~\eqref{eq:clt_G2}. 

It only remains to show that $W$ is independent of $(G_0, G_1)$.
First, for all continuous and bounded functions $\varphi, \psi : \mathbb R\to\mathbb R$,
\ban
\mathbb E[\varphi(W(n))\psi(G_0(n))]
&= \mathbb E\big[\mathbb E[\varphi(W(n))\psi(G_0(n))\mid \bs\nu]\big]
=  \mathbb E\big[\psi(G_0(n))\mathbb E[\varphi(W(n))\mid \bs\nu]\big]\notag\\
&=  \mathbb E\big[\psi(G_0(n))\mathbb E[\varphi(W)]\big]
+  \mathbb E\big[\psi(G_0(n))\big(\mathbb E[\varphi(W(n))\mid \bs\nu]-\mathbb E[\varphi(W)]\big)\big].
\label{eq:truc}
\ean
On the one hand, by linearity (because $\mathbb E[\varphi(W)]$ is a constant), and by~\eqref{eq:clt_G0},
\[\mathbb E\big[\psi(G_0(n))\mathbb E[\varphi(W)]\big]
= \mathbb E[\varphi(W)]\mathbb E\big[\psi(G_0(n))\big] \to\mathbb E[\varphi(W)] \mathbb E[\psi(G_0)].\]
On the other hand,
\[\big|\mathbb E\big[\psi(G_0(n))\big(\mathbb E[\varphi(W(n))\mid \bs\nu]-\mathbb E[\varphi(W)]\big)\big]\big|
\leq \mathbb E\big[\psi(G_0(n))\big|\mathbb E[\varphi(W(n))\mid \bs\nu]-\mathbb E[\varphi(W)]\big|\big]
\to 0,\]
by dominated convergence, because $\psi$ and $\varphi$ are bounded, and by~\eqref{eq:clt_G2}.
Thus, \eqref{eq:truc} implies
\[\mathbb E[\varphi(W(n))\psi(G_0(n))]
\to \mathbb E[\varphi(W)] \mathbb E[\psi(G_0)].\]
Similarly, one can show that, for all continuous and bounded functions 
$\varphi: \mathbb R\to\mathbb R$ and
$\psi : \mathbb R^2\to\mathbb R$,
\[
\mathbb E[\varphi(W(n))\psi(G_0(n),G_1(n))]
\to \mathbb E[\varphi(W)] \mathbb E[\psi(G_0,G_1)].
\]
This implies that~\eqref{eq:clt_G0}, \eqref{eq:clt_G1},
and~\eqref{eq:clt_G2} hold jointly with $W$ independent of 
$(G_0,G_1)$, as desired.
\end{proof}

{\bf Remark:} Note that we cannot say much about the distribution of
$\Lambda = G + W$ since the two Gaussians $G_0$ and $G_1$ (see~\eqref{eq:Lambda}) might be
correlated.

\subsection{The height profile: proof of Theorem \ref{th:profile}}
\label{sec:height_profile}

We now turn to  the case $k=2$
of Theorem \ref{th:profile}, and we write $(u_n,v_n)$ for the two
uniformly random vertices in $\tau_n$ rather than 
$(u^{\sss (1)}_n,u^{\sss (2)}_n)$. 
We follow a strategy
similar to that of Section~\ref{sec:typical_height},  defining a
sequence
$(\tilde \tau_n, \tilde u_n, \tilde v_n)_{n\geq 0}$  
such that for each $n\geq0$, the triple 
$(\tilde \tau_n, \tilde u_n, \tilde v_n)$  
is a distributional copy of $(\tau_n, u_n, v_n)$.  

First,  let $\tilde \tau_0 = \{\varnothing\}$ and 
$\tilde u_0=\tilde v_0 = \varnothing$.
Then, for all $n\geq 0$, given $(\tilde\tau_n, \tilde u_n, \tilde v_n)$, we first sample
$K_1(n+1)$ and $K_2(n+1)$, two independent Bernoulli-distributed random
variables of parameter $1/(2|\tilde\tau_n|+1)$, 
and $L_1(n+1)$ and $L_2(n+1)$, two independent Bernoulli-distributed
random variables of parameter $1/(|\tilde \tau_n|+1)$. 
Finally, we sample $(\alpha_n)_{n\geq 1}$ and $(\beta_n)_{n\geq 1}$, two independent sequences of random variables, uniformly distributed on $\{1, 2\}$.
\begin{enumerate}
\item We let $\tilde\nu_{n}$ be a node taken uniformly at random among the nodes of $\tilde\tau_n$;
\item If $\tilde\nu_n = \varnothing$, then we define 
\[\tilde\tau_{n+1} = \{\varnothing\}\cup\{1w \colon w\in\tilde\tau_n\}\cup\{2w\colon w\in\tilde\tau_n\}.\]
Furthermore, 
\begin{itemize}
\item if $K_1(n+1) = 1$, then we set $\tilde u_{n+1} = \varnothing$, and 
\item if $K_1(n+1) = 0$, then we set $\tilde u_{n+1} = \alpha_{n+1}\tilde u_n$.
\item if $K_2(n+1) = 1$, then we set $\tilde v_{n+1} = \varnothing$, and 
\item if $K_2(n+1) = 0$, then we set $\tilde v_{n+1} = \beta_{n+1}\tilde v_n$.
\end{itemize}
\item If $\tilde\nu_n \neq\varnothing$, then
\begin{itemize}
\item if $L_1(n+1) = L_2(n+1) = 1$, then we set $\tilde\tau_{n+1} = \tilde\tau_n\cup \{\tilde u_n i\}$ and $\tilde u_{n+1} = \tilde v_{n+1} = \tilde u_n i$, where $i = \min\{j\geq 1\colon \tilde u_n j\notin \tilde\tau_n\}$;
\item if $L_1(n+1) = 1$, 
$\tilde u_n\neq\varnothing$ 
and $L_2(n+1) = 0$, then we set $\tilde\tau_{n+1} = \tilde\tau_n\cup \{\tilde u_n i\}$ and $\tilde u_{n+1} = \tilde u_n i$, where  $i = \min\{j\geq 1\colon \tilde u_n j\notin \tilde\tau_n\}$; we also set $\tilde v_{n+1} = \tilde v_n$.
\item 
if $L_1(n+1) = 1$, $\tilde u_n=\varnothing$ and $L_2(n+1) = 0$, 
then we set $\tilde\tau_{n+1} = \tilde\tau_n\cup \{\tilde \nu_n i\}$
and $\tilde u_{n+1} = \tilde \nu_n i$, where  $i = \min\{j\geq 1\colon
\tilde \nu_n j\notin \tilde\tau_n\}$; we also set $\tilde v_{n+1} =
\tilde v_n$. 
\item if $L_1(n+1) = 0$, $L_2(n+1) = 1$ and 
$\tilde v_n\neq\varnothing$, then we set $\tilde\tau_{n+1} = \tilde\tau_n\cup \{\tilde v_n i\}$ and $\tilde v_{n+1} = \tilde v_n i$, where  $i = \min\{j\geq 1\colon \tilde v_n j\notin \tilde\tau_n\}$; we also set $\tilde u_{n+1} = \tilde u_n$.
\item 
if $L_1(n+1) = 0$, $L_2(n+1) = 1$ and $\tilde v_n=\varnothing$, then
we set $\tilde\tau_{n+1} = \tilde\tau_n\cup \{\tilde \nu_n i\}$ and
$\tilde v_{n+1} = \tilde \nu_n i$, where  $i = \min\{j\geq 1\colon
\tilde \nu_n j\notin \tilde\tau_n\}$; we also set $\tilde u_{n+1} =
\tilde u_n$. 
\item if $L_1(n+1) = L_2(n+1) = 0$, then we set $\tilde\tau_{n+1} = \tilde\tau_n\cup \{\tilde \nu_n i\}$, where  $i = \min\{j\geq 1\colon \tilde \nu_n j\notin \tilde\tau_n\}$, $\tilde u_{n+1} = \tilde u_n$, and $\tilde v_{n+1} = \tilde v_n$.
\end{itemize}
\end{enumerate}

Note that, with this definition, $(\tilde \tau_n, \tilde u_n)_{n\geq
  0}$ is the same process as in Section~\ref{sec:typical_height}.
Recall that, in that process, we see some ``resets'' at the root
at doubling-times when also  $K_1(n+1) = 1$, 
and some ``jumps'' at non-doubling times when $L_1(n+1) = 1$ and
$\tilde u_n = \varnothing$, 
while otherwise 
$\tilde u_{n+1}$ is either $\tilde u_n$ or a child of $\tilde u_n$. 
The evolution of $\tilde v_n$ is a bit more complex as it can 
reset at the root  (a doubling-times when $K_2(n+1)= 1$), 
jump (at non-doubling times when also $L_2(n+1) = 1$ 
and $\tilde v_n = \varnothing$), 
jump to $\tilde u_{n+1}$ (at non-doubling times when $L_1(n+1) = L_2(n+1) = 1$), 
and otherwise, $\tilde v_{n+1}$ is either $\tilde v_n$ or a child of $\tilde v_n$. 

\begin{lemma}
For all $n\geq 0$, $(\tilde \tau_n, \tilde u_n, \tilde v_n) \eqd (\tau_n, u_n, v_n)$.
\end{lemma}

\begin{proof}
By induction.
\end{proof}

\begin{proof}[Proof of Theorem \ref{th:profile} for $k=2$]
Again, we identify  $(\tau_n, u_n, v_n) $ with its distributional copy
$(\tilde \tau_n, \tilde u_n, \tilde v_n)$ and drop the tilde from the
notation.  
We first let $n_1 = \max\{n\geq 0\colon L_1(n+1) = 1\text{ and }\tilde u_n = \varnothing\}$
and $n_2 = \max\{n\geq 0\colon L_2(n+1) = 1\text{ and }\tilde v_n = \varnothing\}$.
In the proof of Proposition~\ref{prop:typ_height}, we have shown that $n_1<\infty$ almost surely. Similarly, $n_2<\infty$ almost surely.
We let $n_0 = \max\{n_1, n_2\}$ (and this random integer is almost
surely finite).
We let $R_1(n)$ (resp.\ $R_2(n)$) be the last time before (or at) time
$n$ when $\nu_{i-1} = \varnothing$ and 
$L_1(i) = 1$ (resp.\ $L_2(i) = 1$).
Finally, we let $r_1(n) = \max(R_1(n), n_1)$ and $r_2(n) =
\max(R_2(n), n_2)$. 
With these definitions, we have
\ban
| u_n| 
&= |u_{n_1}|{\bf 1}_{r_1(n) = n_1} +
 \sum_{i=r_1(n)+1}^n {\bf 1}_{ \nu_{i-1} =
  \varnothing} 
+ \sum_{i=r_1(n)+1}^n  {\bf 1}_{ \nu_{i-1} \neq \varnothing} L_1(i)\notag\\
&= \sum_{i=1}^n {\bf 1}_{ \nu_{i-1} = \varnothing} 
+ \sum_{i=1}^n   {\bf 1}_{ \nu_{i-1} \neq \varnothing}L_1(i) + \mathcal O(1),\label{eq:height_u}
\ean
as in Section~\ref{sec:typical_height}. 
Now, we let $S(n)$ be the last time before (or at) time $n$ 
when $\nu_{i-1} \neq \varnothing$ and $L_1(i) = L_2(i) = 1$.
This is the last time before time $n$ when $ v_n$ jumped to join
$u_n$ at a non-doubling time.  
If $S(n) < r_2(n)$, then $v_n$ has 
either reset to $\varnothing$ or jumped to a random node 
since last jumping  to join $u_n$, so
\[
| v_n| =  |v_{n_2}|{\bf 1}_{r_2(n) = n_2} +
\sum_{i=r_2(n)+1}^n {\bf 1}_{\nu_{i-1} = \varnothing} 
+ \sum_{i=r_2(n)+1}^n L_2(i).
\]
If $r_2(n) < S(n)$ (note that they cannot be equal, by definition), then
\[
|v_n| = |u_{S(n)}| + \sum_{i=S(n)+1} {\bf 1}_{ \nu_{i-1} =
  \varnothing} 
+ \sum_{i=S(n)+1}  {\bf 1}_{ \nu_{i-1} \neq \varnothing}  L_2(i).
\]
To summarise, if we let $S_2(n) = r_2(n)\vee S(n)$, then
\be\label{eq:height_v_1}
|v_n| =  |v_{n_2}|{\bf 1}_{r_2(n) = n_2> S(n)} +
|u_{S(n)}|{\bf 1}_{r_2(n)<S(n)} + 
\sum_{i=S_2(n)+1}^n {\bf 1}_{ \nu_{i-1} = \varnothing} + 
\sum_{i=S_2(n)+1}^n   {\bf 1}_{ \nu_{i-1} \neq \varnothing}  L_2(i).
\ee
Note that
\[\mathbb P(\nu_{i-1} \neq \varnothing 
\text{ and }L_1(i) = L_2(i) =  1\mid\c F_{i-1})
= \frac{B_{i-1}}{B_{i-1}+1} \cdot \bigg(\frac1{B_{i-1}+2}\bigg)^{\!2}
\leq \bigg(\frac1{B_{i-1}+2}\bigg)^{\!2}\leq \frac1{i^2},
\]
and thus $\mathbb P(\nu_{i-1} \neq \varnothing 
\text{ and }L_1(i) = L_2(i) =  1)\leq \frac1{i^2}$.
By the Borel--Cantelli lemma, 
almost surely as $n\uparrow\infty$, $S(n)\to S$, where $S$ is an
almost surely finite random variable.
Similarly, as proved in Section~\ref{sec:typical_height}, 
$r_2(n)$ converges almost surely to an almost surely-finite random variable $r_2$.
Thus, 
$S_2(n)\to S_2 = r_2\vee S$ almost surely as $n\uparrow\infty$,
and therefore
\be\label{eq:height_v}
| v_n| =  \sum_{i=1}^n {\bf 1}_{ \nu_{i-1} = \varnothing} +
 \sum_{i=1}^n  {\bf 1}_{ \nu_{i-1} \neq\varnothing}L_2(i) + \mathcal O(1).
\ee
From this point, the rest of the argument is as in the proof of
Proposition \ref{prop:typ_height}:  we write \eqref{eq:height_u}
and \eqref{eq:height_v} as
\be\begin{split}
| u_n| &= 2\sum_{i=1}^n {\bf 1}_{\nu_{i-1} = \varnothing}
- \sum_{i=1}^n \bigg({\bf 1}_{\nu_{i-1} = \varnothing}-\frac1{B_{i-1}+1}\bigg)
+ \sum_{i=1}^n \bigg({\bf 1}_{\nu_{i-1} \neq \varnothing}  L_1(i)-\frac1{B_{i-1}+1}\bigg)
+\c O(1),\\
| v_n| &= 2\sum_{i=1}^n {\bf 1}_{\nu_{i-1} = \varnothing}
- \sum_{i=1}^n \bigg({\bf 1}_{\nu_{i-1} = \varnothing}-\frac1{B_{i-1}+1}\bigg)
+ \sum_{i=1}^n \bigg({\bf 1}_{\nu_{i-1} \neq \varnothing}  L_2(i)-\frac1{B_{i-1}+1}\bigg)
+\c O(1),
\end{split}
\ee
where in each expression, the first sum is handled using 
Proposition \ref{prop:last_doubling} and the others using the
maringale central limit theorem.  Note that this gives
\be
\Big(
\frac{|u_n|- \frac{2\log n}{1+\log 2}}{\sqrt{\frac{\log n}
{1+\log 2}}} ,
\frac{|v_n|- \frac{2\log n}{1+\log 2}}{\sqrt{\frac{\log n}
{1+\log 2}}} \Big)
\Rightarrow (V+W_1,V+W_2),
\ee
where $V=\frac{2G_0}{1+\log 2}-G_1$ as in \eqref{eq:Lambda}
and $W_1,W_2$ are defined as in \eqref{eq:clt_G2} using $L_1$ and
$L_2$, respectively.
\end{proof}

We now briefly comment on the modifications needed for the case
$k\geq3$.    As before, we define a process
$(\tilde \tau_n,\tilde u^{\sss (1)}_n, \tilde u^{\sss (2)}_n,\dotsc, \tilde u^{\sss (k)}_n)$
such that for each $n\geq0$,
$(\tilde \tau_n,\tilde u^{\sss (1)}_n, \dotsc, \tilde u^{\sss (k)}_n)$
has the same distribution as
$( \tau_n,u^{\sss (1)}_n, \dotsc, u^{\sss (k)}_n)$.  In this
process, the triple 
$(\tilde \tau_n,\tilde u^{\sss (1)}_n, \tilde u^{\sss (2)}_n)$
will be the same as the process used for $k=2$ above.  Each
$\tilde u^{\sss (j)}_n$ can be ``reset'' to $\varnothing$ (in the case
$\tilde \nu_n=\varnothing$), 
it can ``jump'' to a random node (at a non-doubling event when $\tilde
u^{\sss (j)}_n=\varnothing$), or 
it can ``jump'' to any of
$\tilde u^{\sss (1)}_n,\dotsc, \tilde u^{\sss (j-1)}_n$ 
(in the case $\tilde \nu_n\neq\varnothing$), it can be replaced by a
``new'' child, or it can remain unchanged, 
these choices being
determined by suitable random variables $K_1(n+1),\dotsc,K_k(n+1)$ 
and $L_1(n+1),\dotsc,L_k(n+1)$.  
In writing expressions such as 
\eqref{eq:height_v_1}, there are many cases to consider, but the
analogs of \eqref{eq:height_u} and \eqref{eq:height_v}
hold for each of 
$\tilde u^{\sss (1)}_n, \dotsc, \tilde u^{\sss (k)}_n$,
and the rest of the argument is as before.

\subsection{Lower bound on the height: 
proof of Proposition \ref{prop:height_LB}}
\label{sec:height_LB}

Let $\f w_n$ denote the leftmost child of the root at height
$\kappa(n)$, where we recall \eqref{eq:kappa_def}
that $\kappa(n)$ is the number of doubling events before time $n$.
Since, by definition, the new node added at non-doubling times
is always added to the right of already existing siblings, $\f w_n$ is
a `copy' of the original root of the tree $\tau_0$.
Let $\f S_n$ denote the subtree rooted at $\f w_n$, let $\f s(n)$
denote the number of nodes in $\f S_n$, and let $\f h(n)$ denote the
height of $\f S_n$.  Note that $\f S_n$ is a
distributional copy of the random recursive tree at time $\f s(n)$.
Clearly, the height $H_n$ of $\tau_n$ satisfies
$H_n\geq \kappa(n)+\f h(n)$.
Moreover, we have the following estimate on the height
of the random recursive tree, which gives
$\f h(n)\sim\mathrm e \log \f s(n)$:

\begin{theorem}[see Pittel~\cite{Pittel}]\label{th:Pittel}
Let $h(n)$ be the height of the $n$-node random recursive tree.
Almost surely as $n\uparrow\infty$, $h(n)\sim \mathrm e\log n$.
\end{theorem}

However, rather than applying Theorem \ref{th:Pittel} directly,
which would require finding estimates on $\f s(n)$,  we use an
embedding of our process $(\tau_n)_{n\geq0}$ into continuous time.
We define a continuous-time
process $(\c T(t))_{t\geq0}$ of growing trees by
first setting  $\c T(0) = \{\varnothing\}$ and then assigning
to every node of the tree a clock that rings at exponential rate of
parameter~1, the   clocks for different nodes being independent.
When a clock rings, if it is the clock associated to $\varnothing$,
then, at that time, we double the tree as done at doubling events in
the discrete time tree; otherwise, we add one child to the node
whose clock rang.  As before, our convention is to add the new
node  to the right of already existing siblings.  

If we let $t_n$ be the time of the $n^{\text{th}}$ ring of a clock,
then $(\mathcal T(t_n))_{n\geq 0} \eqd (\tau_n)_{n\geq 0}$.
Also note that the times at which the tree doubles define a Poisson
point process of intensity~$1$; in particular, the number $D(t)$
of doubling events before time $t$ is distributed as a Poisson of
parameter $t$. 
We now let $\c S(t)$ denote the subtree rooted at the leftmost
node at height $D(t)$.  Thus $\c S(t)$ is the continuous-time version
of $\f S_n$, and it is now simply a Yule process of parameter~$1$.

Note that a random recursive tree can be coupled to a Yule process
so that the random recursive tree equals
the Yule process taken at its successive jump times.
Thus, if $H(t)$ is the height of the Yule process 
$\c S(t)$
at time $t$, then $H(t) = h(|\c S(t)|)$, where 
$|\c S(t)|$ is the number of nodes in the Yule process at time~$t$. 
Thus, by Theorem~\ref{th:Pittel}, almost surely as $t\uparrow\infty$,
\[
\frac{H(t)}{\log |\c S(t)|} \to \mathrm e.
\]
It is well-known that $\mathrm e^{-t}|\c S(t)|$ 
converges almost surely to a standard exponential 
random variable (see, e.g.~\cite[Section~III.5]{AN}), 
which implies $\log |\c S(t)|\sim t$ almost surely as $t\uparrow\infty$.  
Thus, almost surely as $t\uparrow\infty$, 
\[
H(t)\sim \mathrm e\,t.
\]
It follows that the height of $\mathcal T(t)$ is at least
$D(t) + H(t)\sim (1+\mathrm e)t$, almost surely as $t\uparrow\infty$.

\begin{proof}[Proof of Proposition~\ref{prop:height_LB}]
It only remains to translate this lower bound into discrete time. 
For that, we need to understand the asymptotic behaviour of $t_n$,
the times at which $\c T(t)$ grows.
Let $N(t)=|\c T(t)|$ be the number of notes in the tree 
$\c T(t)$ at time~$t$.
At time $t$, the rate at which the next clock rings is $N(t)$,
so we need to understand $N(t)$.

To do this, we will couple $(N(t))_{t\geq 0}$ to a process $(Y(t))_{t\geq0}$
which is the size of a standard Yule process.
Indeed, intuitively $(N(t))_{t\geq0}$ is a Yule process with jumps at
the doubling-times;  the process $(Y(t))_{t\geq0}$ will be defined to
``fill in'' the instantaneous doubling events of $N(t)$ with a Yule
process run for the amount of time it takes to double in size.

To express this more precisely (and we refer to Figure~\ref{fig:warp} for this discussion), 
write $d_1,d_2,\dotsc$ for the
doubling times of $N(t)$, i.e.\ the jump times of the Poisson process
$(D(t))_{t\geq0}$.  For $0\leq t<d_1$, we set $Y(t) = N(t)$.
Then, we let  $(Y(t))_{d_1\leq t<d_1+\ell_1}$ be the size of a Yule
process started at $Y(d_1^-)$ and stopped at time $\ell_1$, defined as
the first time it reaches $2Y(d_1^-)+1$.
Then, for all $n\geq 1$, given $(Y(t))_{t<d_n+\ell_n}$, 
we let $Y(t) = N(t-\ell_n)$ for all $d_n+\ell_n\leq t<d_{n+1}+\ell_n$.
Also, we define $(Y(t))_{d_{n+1} + \ell_n \leq t<d_{n+1}+\ell_{n+1}}$
as a Yule process started at $N(d_{n+1}^-)$ and stopped at time
$\Delta\ell_{n+1} = \ell_{n+1}-\ell_n$, defined as the first time it
hits $2N(d_{n+1}^-)+1$. 
By the strong Markov property, $(Y(t))_{t\geq 0}$ is a Yule process.
Furthermore, by definition, almost surely for all $t\geq 0$,
$Y(t+\ell_{D(t)}) = N(t)$.

\begin{figure}
\begin{center}
\includegraphics[width=12cm]{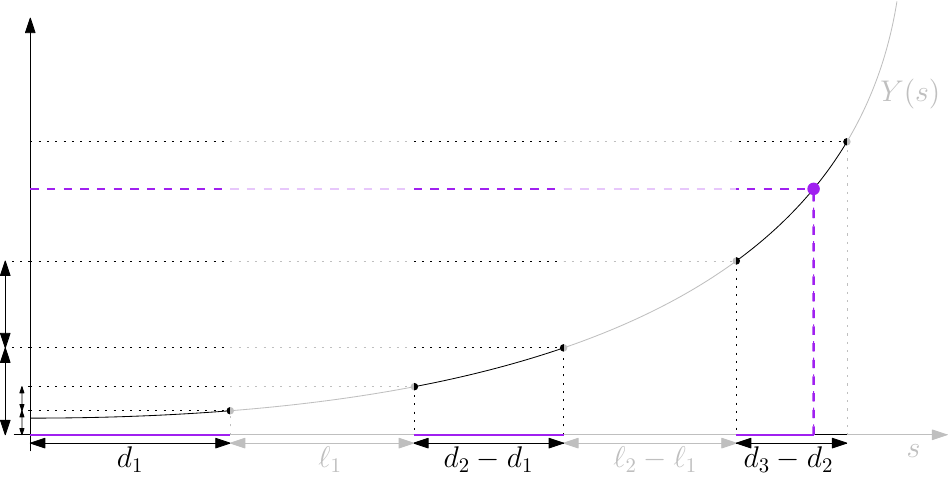}
\end{center}
\caption{
One can see how the process $(N(t))_{t\geq 0}$ can be coupled with a Yule process $(Y(t))_{t\geq 0}$ so that, for all $t\geq 0$, $Y(t+\ell_{D(t)}) = N(t)$, where $D(t)$ is the number of doubling events before time $t$.
The Yule process is the concatenation of the black and grey parts of the curve, 
whilst $(N(t))_{t\geq 0}$ is the curve obtained by only keeping the black parts and gluing them as if time-warps made us skip the intervals of time in grey.
The two pairs of distances highlighted on the left-hand side are such that the two arrows in one pair have the same length: this means that the grey intervals are intervals during which the Yule process doubles in size.
Note that, although both $Y$ and $N$ are jump processes that take value in $\mathbb N$, we have here represented them as continuous curves (with jumps for $N$ when it doubles); this is just for ease of representation.
One can see that, if the total length of all the purple intervals is $t$, then, indeed, $D(t) = 2$ and $N(t)$, which is the value highlighted by large a purple dot, equals $Y(t+\ell_2)$, as claimed in the proof of Proposition~\ref{prop:height_LB}.
}
\label{fig:warp}
\end{figure}

Now $\mathrm e^{-t}Y(t)\to\xi$ almost surely
as $t\to\oo$, where $\xi$ is
exponentially distributed.   It follows that $\log N(t) \sim
t+\ell_{D(t)}$ almost surely.
We now show that $\ell_{D(t)} \sim t\log 2$ almost surely as
$t\uparrow\infty$.
First note that at doubling times we have 
$N(d_i)=Y(d_i+\ell_i)$, thus $\mathrm e^{-(d_i+\ell_i)}N(d_i)\to\xi$
almost surely as $i\to\oo$.  But also 
$N(d_i)=2 N(d_i^-)+1=2 Y(d_i+\ell_{i-1})+1$, so that
\[
\mathrm e^{-(d_i+\ell_i)}N(d_i)
=2\mathrm e^{-(d_i+\ell_{i-1}+\Delta\ell_i)}Y(d_i+\ell_{i-1})
+\mathrm e^{-(d_i+\ell_i)}
\to\xi,
\qquad \text{a.s.\ as }i\to\oo.
\]
Since also $\mathrm e^{-(d_i+\ell_{i-1})}Y(d_i+\ell_{i-1})\to\xi$, 
it follows that $\Delta\ell_i\to\log 2$
almost surely.  Since $D(t)\sim t$  as $t\to\oo$,
\[
\ell_{D(t)}=\sum_{i=1}^{D(t)}\Delta\ell_i\sim t\log 2,
\qquad\text{almost surely},
\]
as claimed.
We thus have that
\be\label{eq:asympt_Nt}
\log N(t) \sim (1+\log 2)t,
\qquad \text{ almost surely as $t\uparrow\infty$}.
\ee 
We claim that
\be\label{eq:liminf_tn}
\liminf_{n\to\oo} \frac{t_n}{\log n} \geq \frac1{1+\log 2},
\qquad \text{ almost surely as }n\uparrow\infty.
\ee
For this, we first note that $t_n\to\oo$ almost surely as $n\to\oo$;
indeed, conditionally on all the $B_i$, we have that $t_n$ is a sum of
independent exponential random variables of rates 
$B_0+1,B_1+1,\dotsc,B_{n-1}+1$, thus the conditional mean of $t_n$
diverges almost surely by Lemma \ref{lem:trick_BC}, while the
conditional variance is bounded since $B_n+1\geq n$ for all $n\geq0$.
Then, using \eqref{eq:asympt_Nt} we have
\[
\log n\leq \log (B_n+1)=\log N(t_n)\sim (1+\log2)t_n,
\qquad\text{ almost surely as }n\to\oo.
\]

From \eqref{eq:liminf_tn}
and the fact that the height of the continuous-time tree $\c T(t)$ 
is asymptotically at least $(1+\mathrm e+o(1))t$ 
(where the $o(1)$-term goes to~0 almost surely as $t\uparrow\infty$), 
we thus get that, almost surely as $n\uparrow\infty$,
\[
H_n\geq (1+\mathrm e+o(1))t_n \geq 
\frac{1+\mathrm e+o(1)}{1+\log 2}\cdot \log n,
\]
as claimed.
\end{proof}

\begin{remark}\label{rk:as-conv}
Using Proposition \ref{prop:cvBn}(d), in fact
$\frac{t_n}{\log n}\to\frac1{1+\log2}$ almost surely as 
$n\to\oo$.   Since $\kappa(n)$ from \eqref{eq:kappa_def}
satisfies $\kappa(n)=D(t_n)\sim t_n$ almost surely, it
follows that $\frac{\kappa(n)}{\log n}\to\frac1{1+\log2}$
almost surely.

Moreover, 
conditionally on $(B_i)_{i\geq 0}$, the expression
$t_n-\sum_{i=0}^{n-1}\frac1{B_i+1}$
is a sum of independent random variables with mean zero and summable variances.
Hence, it is a martingale bounded in $L^2$ which therefore
converges almost surely.  
Using that $t_n\sim \log n/(1+\log 2)$
almost surely as $n\to\oo$, 
it follows that \eqref{eq:equiv_sum_B} can be improved
to an almost sure equivalence:  explicitly,
\[\sum_{i=0}^{n-1}\frac1{B_i+1}\sim\frac{\log n}{1+\log 2},
\qquad\text{ almost surely as }n\to\oo.\]
\end{remark}

\section{A lower bound on the size of the tree doubling everywhere}\label{sec:more_doubling}
As mentioned in the introduction, the model of random recursive tree
that doubles at the root 
is a simplification of a tree that would ``double everywhere''.
We define the random tree $(\tau_n^{\infty})_{n\geq 0}$ recursively as follows: $\tau_0^{\infty} = \{\varnothing\}$ and, for all $n\geq 0$, given $\tau_n^{\infty}$, 
we pick a node $\nu_n$ uniformly at random in $\tau_n^{\infty}$, 
let $t_n$ be the set of (non-strict) descendants of $\nu_n$ in $\tau_n^{\infty}$,
and set
\[\tau_{n+1}^{\infty} 
= \big(\tau_n \setminus t_n\big)
\cup\{\nu_n\}
\cup\{\nu_n 1 w\colon \nu_n w\in\tau_n^{\infty}\}
\cup\{\nu_n 2 w\colon \nu_n w\in\tau_n^{\infty}\}.\]
In other words, at every time step, we pick a node uniformly at random in $\tau_n^{\infty}$, 
remove its subtree (the node and all its descendants) from $\tau_n^{\infty}$ and replace it by two copies of itself as the two subtrees of a new node. 
Note that, by definition, for all $n\geq 0$, $\tau_n^{\infty}$ is
binary, i.e.\ all its words are made on the alphabet $\{1, 2\}$.
We only make the following simple observation about this model:
\begin{proposition}\label{prop:double_everywhere}
For all $n\geq1$ we have
$\mathbb E[|\tau_n^{\infty}|]\geq \frac {n-1}2\log_2(\frac {n-1}{\mathrm e})$.
\end{proposition}
\begin{proof}
For all $n\geq 0$ and $u\in \tau_n^\infty$, we let $s_n(u)$ be the number of (strict) descendants of nodes $u$ in $\tau_n^\infty$. 
Also recall that $|u|$ is the height of node~$u$ (i.e.\ the number of strict ancestors of $u$).
Note that, if at step $n+1$, we select node $u\in\tau_n^\infty$, then
$|\tau_{n+1}^\infty| = |\tau_n^\infty| + 2 + s_n(u)$. Indeed, the tree
rooted at $u$ (which contains
$1+s_n(u)$ nodes) is replaced by a node to which are attached two copies of $u$ and its subtree (which contains $1+2(1+s_n(u))$ nodes in total).
Thus, for all $n\geq 0$,
\[\EE[|\tau_{n+1}^{\infty}|\mid \tau_{n}^{\infty}]
= |\tau_{n}^{\infty}|+\frac1{|\tau_{n}^{\infty}|}
\sum_{u\in  \tau_{n}^{\infty}}\big(2+s_n(u)\big)
= |\tau_{n}^{\infty}|+2
+ \frac1{|\tau_{n}^{\infty}|}
\sum_{u\in  \tau_{n}^{\infty}}\sum_{v\prec u} 1
=|\tau_{n}^{\infty}|+2+\frac1{|\tau_{n}^{\infty}|}
\sum_{v\in  \tau_{n}^{\infty}}|v|.\]
The last term is the expected height of a node chosen uniformly at random in
$\tau_{n}^{\infty}$.
Since $\tau_{n}^{\infty}$ is a binary tree, for any $k\geq0$ the
number of nodes at height at most $k$ is at most $2^{k+1}$. 
Thus, at least half the nodes of $\tau_{n}^{\infty}$ have height at least 
$\log_2|\tau_{n}^{\infty}|-2\geq \log_2 n-2$.
This implies that
\[\EE[|\tau_{n+1}^{\infty}|]\geq
\EE[|\tau_{n}^{\infty}|]+2+
\tfrac12(\log_2 n-2)\geq
\EE[|\tau_{n}^{\infty}|] + \tfrac12\log_2 n.\]
By induction,
\[\EE[|\tau_{n+1}^{\infty}|]\geq\frac12\sum_{j=1}^n\log_2 j
\geq \tfrac n2\log_2(\tfrac{n}{\mathrm e}).
\qedhere
\]
\end{proof}

\begin{remark}
Proposition~\ref{prop:double_everywhere} says that, in expectation, the size of the tree that ``doubles everywhere'' is superlinear in the number of steps. 
By definition, $|\tau_n^{\infty}|\leq 2^n$, where the upper-bound is attained on the event that all doubling events happen at the root. What is the exact order of $\mathbb E[|\tau_n^{\infty}|]$ as $n\uparrow\infty$?
Can we find asymptotic equivalents for $|\tau_n^{\infty}|$ itself, either in probability, or almost surely?
We leave these as open problems. 
\end{remark}

\bibliographystyle{alpha}
\bibliography{refs_doubling_tree}

\end{document}